\documentclass[a4paper,11pt,twoside, dvipdfmx]{article}
\usepackage{comment}
\usepackage{mathptmx}  
\usepackage{amsmath, amssymb, amsthm}
\usepackage{graphicx}
\numberwithin{equation}{section}

\newcounter{num}
\newcommand{\Rnum}[1]{\setcounter{num}{#1} \Roman{num}}

\newcommand{\1}{\mbox{1}\hspace{-0.25em}\mbox{l}}
\newtheorem{defi}{Definition}
\newtheorem{theo}{Theorem}
\newtheorem{lemm}{Lemma}

\newtheorem{remark}{Remark}

\allowdisplaybreaks[1]

\usepackage{color}

\newcommand{\blue}{\color[rgb]{0,0,1}}

\def\df{\mathrm{d}}
\def\wh#1{\widehat{#1}} 
\def\ol#1{\overline{#1}} 

\title{Statistical inference for discretely sampled stochastic functional differential equations with small noise}
\author{Hiroki Nemoto\footnote{E-mail: hiroki-n17@fuji.waseda.jp}\\
{\it Graduate School of Fundamental Science and Engineering, Waseda University}\vspace{0.2cm}\\
Yasutaka Shimizu\footnote{E-mail: shimizu@waseda.jp} \\ 
{\it Department of Applied Mathematics, Waseda University}
}
\date{\today}

\begin{document}
\maketitle

\begin{abstract} 
\noindent
Estimating parameters of drift and diffusion coefficients for multidimensional stochastic delay equations with small noise are considered. 
The delay structure is written as an integral form with respect to a delay measure. 
Our contrast function is based on a local-Gauss approximation to the transition probability density of the process. We show consistency and asymptotic normality of the minimum-contrast estimator when a small dispersion coefficient $\varepsilon\to 0$ and sample size $n\to\infty$ simultaneously.
\begin{flushleft}
\vspace{1mm}  
\end{flushleft}
\end{abstract}

\section{Introduction} \label{sec:introduction}
Let $(\Omega,\mathcal{F},\{\mathcal{F}_t\},P)$ be a stochastic basis satisfying the usual conditions. We consider a family of $d$-dimensional stochastic functional differential equations (SFDEs): for $\varepsilon\in(0,1]$, 
\begin{align}
\left\{
\begin{array}{l}
\df X_t^\varepsilon=b\big(X_t^\varepsilon,H(X_{t-\cdot}^\varepsilon),\theta_0\big)\,\df t+\varepsilon\sigma\big(X_t^\varepsilon,H(X_{t-\cdot}^\varepsilon),\beta_0\big)\,\df W_t, \quad t\in[0,1];\\
X_t^\varepsilon=\phi^\varepsilon(t), \quad t\in[-\delta, 0); \\
X_0^\varepsilon=\phi^\varepsilon(0)=x_0^\varepsilon, \label{SFDEs}
\end{array}
\right.
\end{align}
 
where $\theta_0=(\alpha_0,\beta_0)\in \mathring{\Theta}$ with $\Theta=\ol{\Theta}_{\alpha}\times\ol{\Theta}_{\beta}$ for open bounded convex subsets $\Theta_{\alpha}$ and $\Theta_{\beta}$ of $\mathbb{R}^p$ and $\mathbb{R}^q$, respectively; 
$b=(b_1,\dots,b_d)\colon\mathbb{R}^d\times\mathbb{R}^d\times\Theta\to\mathbb{R}^d$ and $\sigma=(\sigma_{ij})_{d\times r}\colon\mathbb{R}^d\times\mathbb{R}^d\times\ol{\Theta}_{\beta}\to\mathbb{R}^{d}\otimes \mathbb{R}^r$ are known functions; 
$W_t=(W_t^1, \dots, W_t^r)$ is an $r$-dimensional Wiener process.
Moreover, $\phi^\varepsilon(t)$ is $\mathcal{F}_0$ measurable $\mathbb{R}^d$-valued random variable for each $t \in [-\delta,0]$ and $\varepsilon \in [0,1]$. Letting $C(A; B)$ be the space of continuous functions from $A$ to $B$, 
we also define a functional $H\colon C([0,\delta];\mathbb{R}^d)\to\mathbb{R}^d$ as follows: for a continuous function $F_{t-\cdot}: u\mapsto F_{t-u}$ in $u\in[0,\delta]$, 
$$
H(F_{t-\cdot})=\int_0^{\delta}F_{t-u}\,\mu(\df u), 
$$
where $\mu$ is a finite measure on $[0,\delta]$. 
We call the equation \eqref{SFDEs} {\it stochastic functional delay equations (SFDEs)} since the functional $H$ provides a delay structure in the SDE. Especially, when the measure $\mu$ is a Dirac measure, the SDE is just called a {\it stochastic delay differential equation (SDDE)}. 

We assume that the process $\{X_t^\varepsilon\}$ is observed at regularly spaced time points $\{t_k=k/n\,|\,k=0,\dots,n\}\cup \{-i/n\,|\,i=1,\dots,\lfloor n\delta\rfloor\}$ with the floor function $\lfloor\cdot\rfloor$. 
Our goal is to construct an estimator of $\theta_0$ from discrete observations $\{X_{t_k}^\varepsilon\}\cup \{X_{-i/n}^\varepsilon\}$, 
and to investigate the asymptotic behavior when $\varepsilon \to 0$ as well as $n\to \infty$. 

There are many applications of (deterministic) delay differential equations (DDEs) in biology, epidemiology, and physics. 
For example, Mackey and Glassa \cite{mackey} consider a homogeneous population of mature circulating cells (white blood cell, red blood cell, or platelet). Hu and Wang \cite{hu} take account of the dynamics of controlled mechanical systems, among others. 
Due to those, their corresponding stochastic versions (SDDEs) of those differential equations also has been well investigated. 
Volterra \cite{volterra} considers predator-prey models; 
Guttorp and Kulperger \cite{guttorp} take the effect of random elements into account, and they change Volterra's model from DDEs to SDDEs; 
Fu \cite{fu} considers a stochastic SIR model with delay for an epidemic model, and also, De la Sen and Ibeas \cite{de la sen} consider SE(Is)(Ih)AR model as a COVID-19 model. From a theoretical point of view, Mohammed \cite{mohhamed}, and Arriojas \cite{arriojas} have generalized those models to SFDEs. 

Turnning our attention to statistical inference for SDDEs and SFDEs, there have been many works so far. 
Gunshchin and K\"{u}chler \cite{gunshchin} and K\"{u}chler and Kutoyants \cite{kuchler} study the asymptotic behavior of the maximum likelihood type estimators; K\"{u}chler and S{\o}rensen \cite{kuchler2} consider the pseudo-likelihood estimator for SDDEs; 
K\"{u}chler and Vasil'jev \cite{kuchler3} propose a sequential procedure with a given accuracy in the $L_2$ sense. 
Moreover, Reiss \cite{reiss, reiss2} investigate nonparametric inference for affine SDDEs; 
Ren and Wu \cite{ren} consider least squares estimates for path-dependent McKean-Vlasov SDEs from discrete observations. 
Although all of those are studied in the ergodic context, we are interested in the small noise case; $\epsilon \to 0$, which is useful to justify the validity of the estimators since, in most applications of SFDEs, the ergodicity is often not expected. 
In this paper, we consider a {\it local-Gauss type} contrast function, and show the asymptotic normality of the minimum contrast estimators. 

The paper is organized as follows. In Section \ref{sec:notation}, we make notation and assumptions, and state our main results in Section \ref{sec:main}. 
In Section \ref{sec:simulation}, we provide some numerical studies to support for our results. 
All the mathematical proofs are put in Section \ref{sec:proofs}. 

\section{Notation and assumptions}\label{sec:notation}

\subsection{Notation}
\renewcommand{\labelenumi}{(N\theenumi)}
\newcounter{enumimemory}
\begin{enumerate}
\item $X_t^0$ is the solution of the ordinary differential equations under the true value of the drift parameter for $t\in[0,1]$: $X_0^0=\phi(0)=x_0$ and 
\begin{align}
\left\{
\begin{array}{ll}
\df X_t^0 = b\left(X_t^0,H(X_{t-\cdot}^0),\theta_0\right)\,\df t, & t\in[0,1];\notag\\
X_t^0 = \phi(t), & t\in[-\delta,0), \label{ODEs}
\end{array}
\right.
\end{align}
where $\phi\in C([-\delta, 0];\mathbb{R}^d)$ and $x_0$ is constant. As for the existence and uniqueness of the solution, see the proof of Theorem 3.7 and Remark 3.8 by Smith \cite{smith}.
\item For matrix $A$, the $(i,j)$th element is written by $A^{ij}$, ansd that 
\begin{equation*}
|A|^2=\mathrm{tr}\left(AA^\top\right),
\end{equation*}
where $A^\top$ is the transpose of $A$ and $\mathrm{tr}\left(AA^\top\right)$ is the trace of $AA^\top$.
\item For multi-index $m=(m_1,\dots,m_k)$, a derivative operator in $z\in \mathbb{R}^k$ is given by
\begin{equation*}
\partial_z^m:=\partial_{z_1}^{m_1}\cdots \partial_{z_k}^{m_k},\qquad \partial_{z_i}^{m_i}\colon=\left(\partial/\partial_{z_i}\right)^{m_i}. 
\end{equation*}

\item Let $C^{j,k,l}(\mathbb{R}^d\times\mathbb{R}^d\times\Theta;\mathbb{R}^N)$ be the space of all functions $f$ satisfying that $f(x,y,\theta)$ is a $\mathbb{R}^N$-valued function on $\mathbb{R}^d\times\mathbb{R}^d\times\Theta$ which $j$, $k$ and $l$ times continuously differentiable with respect to $x$, $y$ and $\theta$, respectively.
\item $C_{\uparrow}^{j,k,l}(\mathbb{R}^d\times\mathbb{R}^d\times\Theta;\mathbb{R}^N)$ is a class of $C^{j,k,l}(\mathbb{R}^d\times\mathbb{R}^d\times\Theta;\mathbb{R}^N)$ satisfying that
\begin{equation*}
\sup_{\theta\in\Theta}|\partial^\mu_\theta\partial^\nu_y\partial^\xi_x f(x,y,\theta)|\leq C(1+|x|+|y|)^\lambda,
\end{equation*}
for universal positive constants $C$ and $\lambda$, where for $M=\mathrm{dim}(\Theta)$, $\mu=(\mu_1,\dots,\mu_M)$, $\nu=(\nu_1,\dots,\nu_d)$ and $\xi=(\xi_1,\dots,\xi_d)$ are multi-indices with $0\leq\sum_{i=1}^M\mu_i\leq l$, $0\leq\sum_{i=1}^d\nu_i\leq k$ and $0\leq\sum_{i=1}^d\xi_i\leq j$, respectively.
\item Denote by $G(p)$ the set of all permutations on $\{1,\dots,p\}$.
\item For elements $\left\{b^i\right\}$ and $\left\{[\sigma\sigma^\top]^{ij}\right\}$, we denote by
\begin{equation*}
b^i_{t,H}(\theta):=b^i\left(X_t^\varepsilon,H\left(X_{t-\cdot}^\varepsilon\right),\theta\right),\quad 
[\sigma\sigma^\top]_{t,H}^{ij}(\beta):=[\sigma\sigma^\top]^{ij}\left(X_t^\varepsilon,H\left(X_{t-\cdot}^\varepsilon\right),\beta\right).
\end{equation*}
\item Denote by
$\Delta_kX^\varepsilon:= X_{t_k}^\varepsilon-X_{t_{k-1}}^\varepsilon$ and 
\begin{equation*}
B\left(X_t^0, \theta_0, \theta\right):= b\left(X_t^0,H\left(X_{t-\cdot}^0\right),\theta_0\right)-b\left(X_t^0,H\left(X_{t-\cdot}^0\right),\theta\right).
\end{equation*}
\setcounter{enumimemory}{\value{enumi}}
\end{enumerate}

\subsection{Assumptions}
We make the following assumptions:
\renewcommand{\labelenumi}{(A\theenumi)}
\begin{enumerate}
\item There exists a constant $K>0$ such that \label{assu1}
\begin{align*}
|b(x,y,\theta)-b(\tilde{x},\tilde{y},\theta)|+|\sigma(x,y,\theta)-\sigma(\tilde{x},\tilde{y},\theta)|&\leq K\left(|x-\tilde{x}|+|y-\tilde{y}|\right),\\
|b(x,y,\theta)|+|\sigma(x,y,\theta)|&\leq K\left(1+|x|+|y|\right),
\end{align*}
for each $x,\tilde{x},y,\tilde{y}\in\mathbb{R}^d$ and $\theta\in\Theta$.

\item For any $p\geq 1$, \label{assu2}
\begin{equation*}
\sup_{\varepsilon\in(0,1]}E\left[\sup_{t\in[-\delta,0]}|\phi^\varepsilon(t)|^p\right]<\infty,
\end{equation*}
and there exists a constant $K_1,K_2>0$ such that 
\begin{equation*}
E\left[|\phi^\varepsilon(t)-\phi^\varepsilon(s)|^p\right]\leq K_1|t-s|^p + K_2\varepsilon^p|t-s|^{p/2}. 
\end{equation*} 
Moreover, as $\varepsilon \to 0$,
\begin{equation*}
E\left[\sup_{-\delta \leq t\leq 0}|\phi^{\varepsilon}(t)-\phi(t)|^p\right]=O(\varepsilon^p).
\end{equation*}

\item $b\left(\cdot,\cdot,\cdot\right)\in C_{\uparrow}^{2,1,3}\left(\mathbb{R}^d\times\mathbb{R}^d\times\Theta;\mathbb{R}^d\right)$, $\sigma(\cdot,\cdot,\cdot)\in C_{\uparrow}^{2,1,3}\left(\mathbb{R}^d\times\mathbb{R}^d\times\Theta;\mathbb{R}^d\otimes \mathbb{R}^r\right)$. \label{assu3}

\item The matrix $[\sigma\sigma^\top]\left(x,y,\beta\right)$ is positive definite for each $x,y\in\mathbb{R}^d$ and $\beta\in\ol{\Theta}_{\beta}$, and that 
\[
\inf_{x,y\in\mathbb{R}^d,\beta\in\ol{\Theta}_{\beta}}\det[\sigma\sigma^\top]\left(x,y,\beta\right)>0. 
\]
Moreover,
$[\sigma\sigma^\top]^{-1}(\cdot,\cdot,\cdot)\in C_{\uparrow}^{1,1,3}\left(\mathbb{R}^d\times\mathbb{R}^d\times\ol{\Theta}_{\beta};\mathbb{R}^d\otimes \mathbb{R}^d\right)$. \label{assu4}

\item\label{assu5}  If $\theta\neq\theta_0$ then $b(X_t^0,H(X_{t-\cdot}^0),\theta)\neq b(X_t^0,H(X_{t-\cdot}^0),\theta_0)$;  
If $\beta\neq\beta_0$ then $[\sigma\sigma^\top](X_t^0,H(X_{t-\cdot}^0),\beta)\neq [\sigma\sigma^\top](X_t^0,H(X_{t-\cdot}^0),\beta_0)$, 
 for at least one value of $t$, respectively. 

\item \label{assu6} The matrix
\begin{equation*}
I(\theta_0)=
\left(
\begin{array}{cc}
\left(I_b^{ij}(\theta_0)\right)_{1\leq i,j\leq p} & 0\\
0 & \left(I_\sigma^{ij}(\theta_0)\right)_{1\leq i,j\leq q}\\
\end{array}
\right),
\end{equation*}
is positive definite, where
\small{
\begin{align*}
I_b^{ij}(\theta_0)&=\int_0^1\bigg(\frac{\partial}{\partial\alpha_i}b\left(X_{s}^0,H\left(X_{s-\cdot}^0\right),\theta_0\right)\bigg)^{\top}[\sigma\sigma^\top]^{-1}\left(X_s^0,H\left(X_{s-\cdot}^0\right),\beta_0\right)\bigg(\frac{\partial}{\partial\alpha_j}b\left(X_{s}^0,H\left(X_{s-\cdot}^0\right),\theta_0\right)\bigg)\,\df s,\\
I_\sigma^{ij}(\theta_0)&=\frac{1}{2}\int_0^1\mathrm{tr}\Bigg[\bigg(\frac{\partial}{\partial\beta_i}[\sigma\sigma^{\top}]\bigg)[\sigma\sigma^\top]^{-1}\bigg(\frac{\partial}{\partial\beta_j} [\sigma\sigma^\top]\bigg)[\sigma\sigma^\top]^{-1}\left(X_s^0,H\left(X_{s-\cdot}^0\right),\beta_0\right)\Bigg]\,\df s. 
\end{align*}
}
\end{enumerate}
\begin{remark}
\upshape
Although the assumption (A\ref{assu4}) seems a bit restrictive, it is the same assumption as [A3'] in Gloter and S\o rensen \cite{gloter}.
\end{remark}

\section{Main theorems}\label{sec:main}

For estimation of $\theta\in \Theta$ in \eqref{SFDEs}, we consider the following local-Gauss type contrast function:
\begin{equation*}
U_{n,\varepsilon}(\theta)=\sum_{k=1}^{n}\left\{\log\det\Xi_{k-1}(\beta)+\frac{n}{\varepsilon^2}P_k^\top(\theta)\Xi_{k-1}^{-1}(\beta)P_k(\theta)\right\},
\end{equation*}
where
\begin{equation*}
P_k(\theta)=\Delta_k X-\frac{1}{n}b\left(X_{t_{k-1}},H_n(X_{t_{k-1}-\cdot}),\theta\right),\quad \Xi_{k-1}(\beta)=[\sigma\sigma^{\top}]\left(X_{t_{k-1}},H_n(X_{t_{k-1}-\cdot}),\beta\right),
\end{equation*}
\begin{equation*}
H_n(X_{t_{k-1}-\cdot})=\sum_{i=1}^{\lfloor n\delta \rfloor}\left\{X_{t_{k-1}-(i-1)/n}~\mu \big([(i-1)/n, i/n)\big)\right\}+X_{t_{k-1}-\delta_n}~\mu \big([\delta_n, \delta]\big),
\end{equation*}
and $\delta_n := \lfloor n\delta \rfloor/n$. 
\begin{defi}
A minimum contrast estimator $\wh{\theta}_{n,\varepsilon}=(\wh{\alpha}_{n,\varepsilon},\wh{\beta}_{n,\varepsilon})$ is defined as
\begin{equation*}
U_{n,\varepsilon}(\wh{\theta}_{n,\varepsilon})=\inf_{\theta\in\Theta}U_{n,\varepsilon}(\theta).
\end{equation*}
\end{defi}

The consistency of our estimator $\wh{\theta}_{n,\varepsilon}$ is given as follows.
\begin{theo}\label{consistency}
Suppose the assumptions {\upshape(A\ref{assu1})--(A\ref{assu5})}.  Then we have
\begin{equation*}
\wh{\theta}_{n,\varepsilon} \xrightarrow{P}\theta_0,
\end{equation*}
if $(\sqrt{n}\varepsilon)^{-1}\to0$ as $\varepsilon\to0$ and $n\to\infty$. 
\end{theo}
\noindent The next theorem gives the asymptotic normal distribution of $\wh{\theta}_{n,\varepsilon}$.
\begin{theo}\label{asymptotic normal}
Suppose the assumptions {\upshape(A\ref{assu1})--(A\ref{assu6})}. Then we have
\begin{equation*}
\left(
\begin{array}{rr}
\varepsilon^{-1}(\wh{\alpha}_{n,\varepsilon}-\alpha_0) \\
\sqrt{n}(\wh{\beta}_{n,\varepsilon}-\beta_0) \\
\end{array}
\right)
\xrightarrow{d} N\left(0,I^{-1}(\theta_0)\right),
\end{equation*}
$(\sqrt{n}\varepsilon)^{-1}\to0$ as $\varepsilon\to0$ and $n\to\infty$. 
\end{theo}

\section{Simulations}\label{sec:simulation}
We consider the following $2$ -dimensional SFDE:
\begin{equation*}
\left\{
\begin{aligned}
\df X_t^{(1)}=\alpha_1H(X_{t-\cdot}^{(2)})\,\df t+\varepsilon\beta_1\sqrt{1+\left(H(X_{t-\cdot}^{(2)})\right)^2}\,\df W_t^1,\\
\df X_t^{(2)}=\alpha_2H(X_{t-\cdot}^{(1)})\,\df t+\varepsilon\beta_2\sqrt{1+\left(H(X_{t-\cdot}^{(1)})\right)^2}\,\df W_t^2,
\end{aligned}
\right.
\end{equation*}
for $t\in[0,1]$,
\begin{equation*}
\left\{
\begin{aligned}
\df X_t^{(1)}=5X_t^{(2)}\,\df t+7\varepsilon\sqrt{1+\left(X_t^{(2)}\right)^2}\,\df W_t^1,\\
\df X_t^{(2)}=6X_t^{(1)}\,\df t+8\varepsilon\sqrt{1+\left(X_t^{(1)}\right)^2}\,\df W_t^2,
\end{aligned}
\right.
\end{equation*}
for $t\in[-\delta,0]$, where $\delta=1/10$, $\left(X_{-\delta}^{(1)},X_{-\delta}^{(2)}\right)=(1,2)$ and $H(X_{t-\cdot})=X_{t-\delta}$. 
In this example, the estimator is given explicitly as follows:
\begin{equation*}
\wh{\theta}_{n,\varepsilon}=(\wh{\alpha}_{n,\varepsilon},\wh{\beta}_{n,\varepsilon})=(\wh{\alpha}_{n,\varepsilon,1},\wh{\alpha}_{n,\varepsilon,2},\wh{\beta}_{n,\varepsilon,1},\wh{\beta}_{n,\varepsilon,2}),
\end{equation*}
where 
\begin{equation*}
\wh{\alpha}_{n,\varepsilon,1} = \frac{n\displaystyle\sum_{k=1}^{n}\frac{H_n(X_{t_{k-1}-\cdot}^{(2)})}{1+\left(H_n(X_{t_{k-1}-\cdot}^{(2)})\right)^2}}{\displaystyle\sum_{k=1}^{n}\frac{\left(H_n(X_{t_{k-1}-\cdot}^{(2)})\right)^2}{1+\left(H_n(X_{t_{k-1}-\cdot}^{(2)})\right)^2}}, \quad 
\wh{\alpha}_{n,\varepsilon,2} = \frac{n\displaystyle\sum_{k=1}^{n}\frac{H_n(X_{t_{k-1}-\cdot}^{(1)})}{1+\left(H_n(X_{t_{k-1}-\cdot}^{(1)})\right)^2}}{\displaystyle\sum_{k=1}^{n}\frac{\left(H_n(X_{t_{k-1}-\cdot}^{(1)})\right)^2}{1+\left(H_n(X_{t_{k-1}-\cdot}^{(1)})\right)^2}},
\end{equation*}
\begin{equation*}
\wh{\beta}_{n,\varepsilon,1} = \varepsilon^{-1}\displaystyle\sqrt{\sum_{k=1}^{n}\frac{\left(\Delta_{t_{k-1}}X^{(1)}-\frac{\wh{\alpha}_{n,\varepsilon,1}}{n}H_n(X_{t_{k-1}-\cdot}^{(2)})\right)^2}{1+\left(H_n(X_{t_{k-1}-\cdot}^{(2)})\right)^2}}, \quad 
\wh{\beta}_{n,\varepsilon,2} = \varepsilon^{-1}\displaystyle\sqrt{\sum_{k=1}^{n}\frac{\left(\Delta_{t_{k-1}}X^{(2)}-\frac{\wh{\alpha}_{n,\varepsilon,1}}{n}H_n(X_{t_{k-1}-\cdot}^{(1)})\right)^2}{1+\left(H_n(X_{t_{k-1}-\cdot}^{(1)})\right)^2}},
\end{equation*}
and
$H_n(X_{t-\cdot})=X_{t-\delta_n}$.

In the experiments, we generate discrete samples $\{X_{t_{k}}\}_{k=1}^n$ and $\{X_{-i/n}\}_{i=1}^{n\delta_n}$ by the Euler-Maruyama method (see Buckwar \cite{backwar}). We show means and standard deviations of estimators through 1000 times replications according to several values of $(n,\varepsilon)$ in Tables \ref{table1}--\ref{table3}, which illustrate the consistency of our estimator. 

We also show the results of normal Q-Q plot in the ideal case where $(n,\varepsilon) = (10000, 0.01)$ in Figures \ref{fig1}-\ref{fig4}, which illustrate the asymptotic normality of each marginal of $\wh{\theta}_{n,\varepsilon}$. Moreover, Figure \ref{fig5} shows that the distribution of the bilinear form of the estimator follows the $\chi^2(4)$-distribution, which illustrates the (joint) asymptotic normality of $\wh{\theta}_{n,\varepsilon}$.

\begin{table}[h]
\caption{Mean and standard deviation of the estimator $\wh{\theta}_{n,\varepsilon}$ through 1000 experiments as $\varepsilon=0.1$.}
\label{table1}
\centering
\begin{tabular}{ccccc}
\hline
$\varepsilon=0.1$ & $n=100$ & $n=1000$ & $n=10000$ & True\\
\hline
$\wh{\alpha}_{n,\varepsilon,1}$ (s.d.)&1.01707(0.30698)&1.00405(0.31741)&1.00042(0.30145)&1.0\\

$\wh{\alpha}_{n,\varepsilon,2}$ (s.d.)&2.00650(0.42907)&1.99539(0.42342)&2.01553(0.43601)&2.0\\

$\wh{\beta}_{n,\varepsilon,1}$ (s.d.)&2.98545(0.20686)&2.99715(0.06410)&2.99885(0.02126)&3.0\\

$\wh{\beta}_{n,\varepsilon,2}$ (s.d.)&3.96573(0.28401)&3.99143(0.08755)&4.00007(0.02858)&4.0\\
\hline
\end{tabular}
\end{table}

\begin{table}[h]
\caption{Mean and standard deviation of the estimator $\wh{\theta}_{n,\varepsilon}$ through 1000 experiments as $\varepsilon=0.03$.}
\label{table2}
\centering
\begin{tabular}{ccccc}
\hline
$\varepsilon=0.03$ & $n=100$ & $n=1000$ & $n=10000$ & True\\
\hline
$\wh{\alpha}_{n,\varepsilon,1}$ (s.d.)&0.99851(0.09238)&1.00105(0.09122)&0.99778(0.09053)&1.0\\

$\wh{\alpha}_{n,\varepsilon,2}$ (s.d.)&2.00184(0.12508)&2.00489(0.12605)&2.00215(0.12024)&2.0\\

$\wh{\beta}_{n,\varepsilon,1}$ (s.d.)&2.99172(0.21613)&2.99504(0.06489)&2.99993(0.02170)&3.0\\

$\wh{\beta}_{n,\varepsilon,2}$ (s.d.)&3.98198(0.27994)&3.99969(0.08832)&3.9997(0.02734)&4.0\\
\hline
\end{tabular}
\end{table}

\begin{table}[h]
\caption{Mean and standard deviation of the estimator $\wh{\theta}_{n,\varepsilon}$ through 1000 experiments as $\varepsilon=0.01$.}
\label{table3}
\centering
\begin{tabular}{ccccc}
\hline
$\varepsilon=0.01$ & $n=100$ & $n=1000$ & $n=10000$ & True\\
\hline
$\wh{\alpha}_{n,\varepsilon,1}$ (s.d.)&1.00230(0.03072)&1.00002(0.03148)&0.99984(0.03081)&1.0\\

$\wh{\alpha}_{n,\varepsilon,2}$ (s.d.)&2.00192(0.04110)&1.99903(0.04135)&1.99951(0.04072)&2.0\\

$\wh{\beta}_{n,\varepsilon,1}$ (s.d.)&2.99120(0.21070)&2.99924(0.06910)&2.99994(0.02195)&3.0\\

$\wh{\beta}_{n,\varepsilon,2}$ (s.d.)&3.98945(0.27986)&3.99760(0.09188)&4.00108(0.02737)&4.0\\
\hline
\end{tabular}
\end{table}

\begin{figure}[htbp]
\center \includegraphics[width = 9cm]{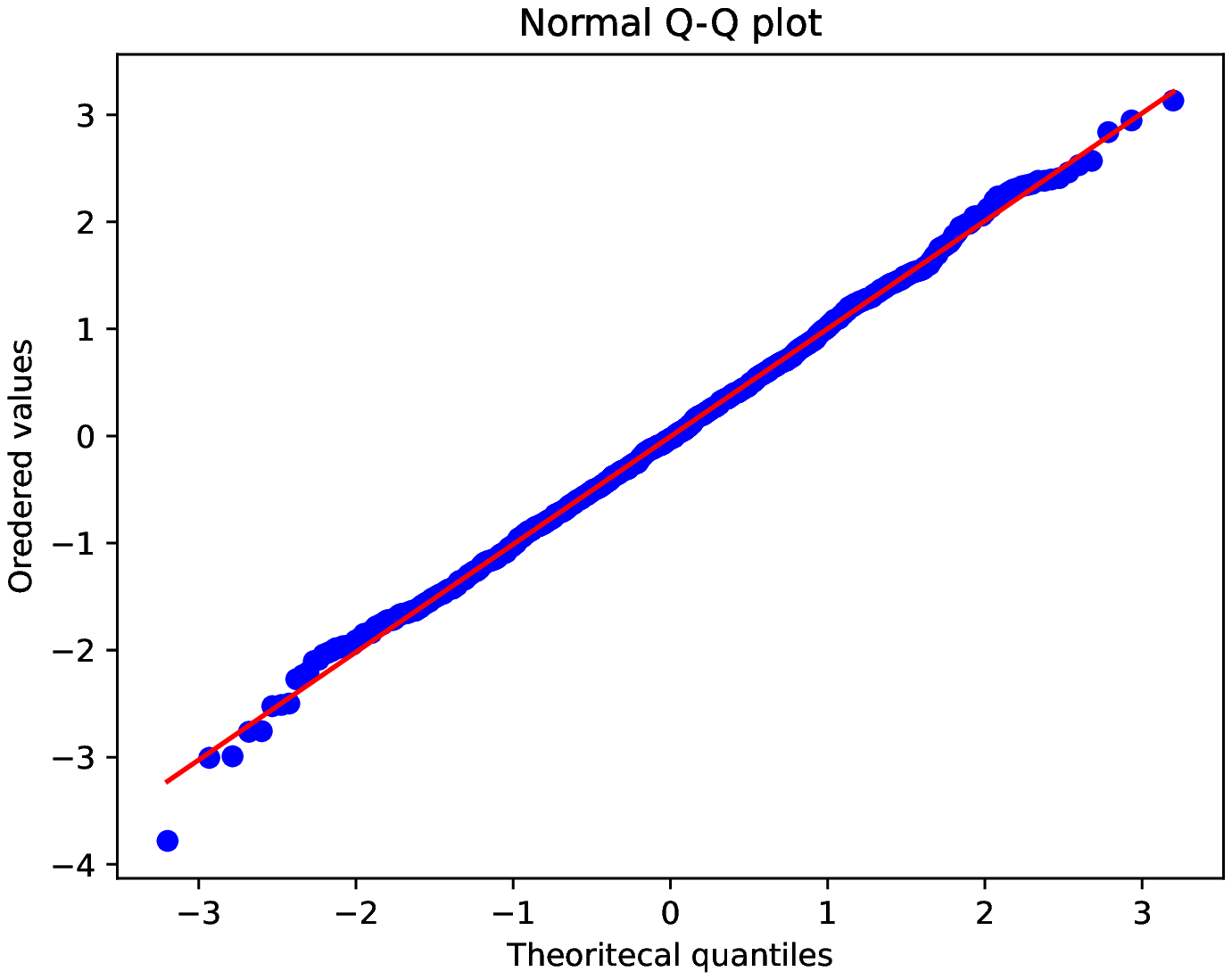}
\caption{Normal Q-Q plot as $\theta=(1.0, 2.0, 3.0, 4.0), \varepsilon=0.01, n=10000$ for $1000$ iterated samples of $\varepsilon^{-1}\sqrt{I_{b}^{11}(\theta_0)}(\wh{\alpha}_{n,\varepsilon,1}-\alpha_1)$.}
\label{fig1}
\end{figure}

\begin{figure}[htbp]
\centering
\includegraphics[width = 9cm]{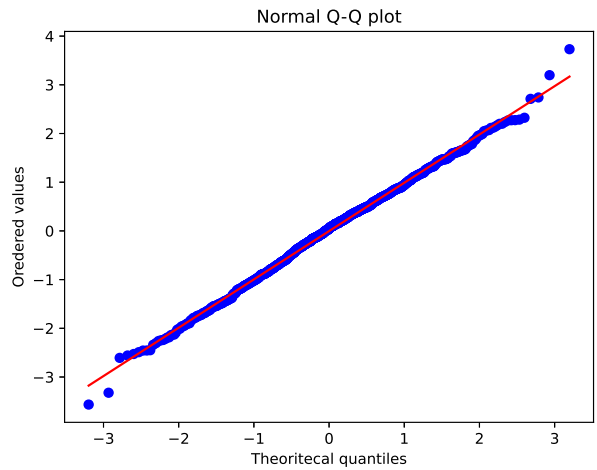}
\caption{Normal Q-Q plot as $\theta=(1.0, 2.0, 3.0, 4.0), \varepsilon=0.01, n=10000$ for $1000$ iterated samples of $\varepsilon^{-1}\sqrt{I_{b}^{22}(\theta_0)}(\wh{\alpha}_{n,\varepsilon,2}-\alpha_2)$.}
\label{fig2}
\end{figure}

\begin{figure}[htbp]
\centering
\includegraphics[width = 9cm]{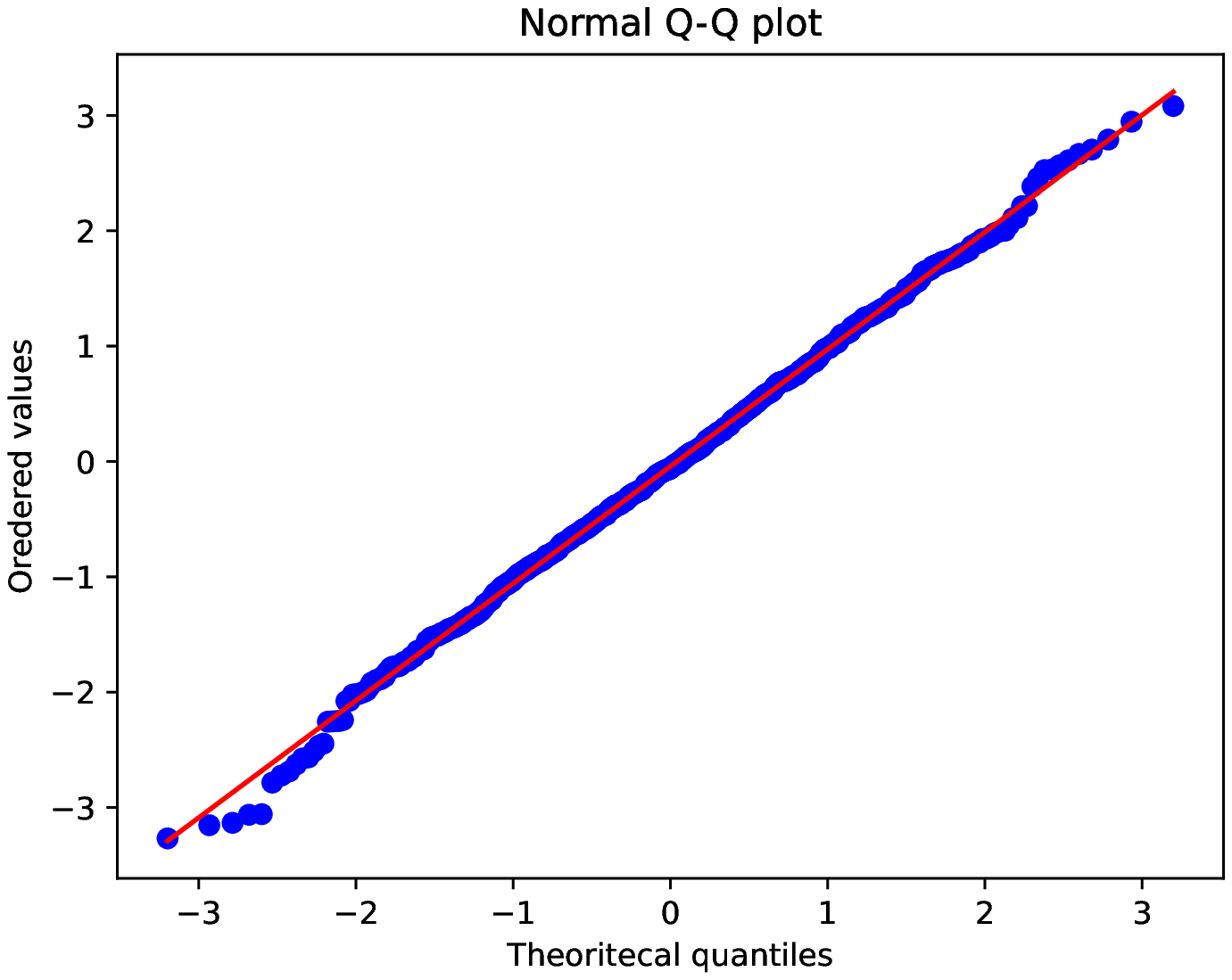}
\caption{Normal Q-Q plot  as $\theta=(1.0, 2.0, 3.0, 4.0), \varepsilon=0.01, n=10000$ for $1000$ iterated samples of $\sqrt{n}\sqrt{I_{\sigma}^{11}(\theta_0)}(\wh{\beta}_{n,\varepsilon,1}-\beta_1)$.}
\label{fig3}
\end{figure}

\begin{figure}[htbp]
\centering
\includegraphics[width = 9cm]{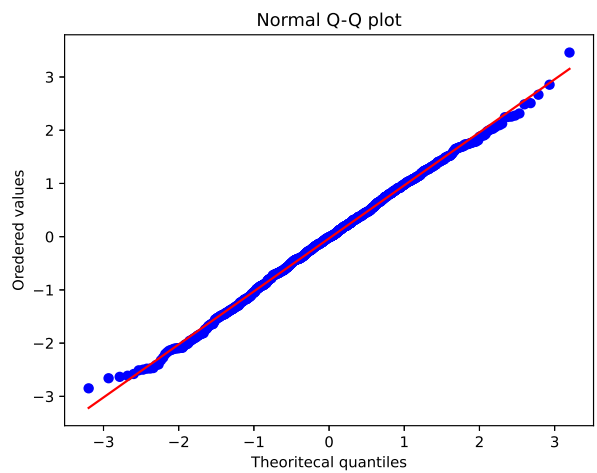}
\caption{Normal Q-Q plot as $\theta=(1.0, 2.0, 3.0, 4.0), \varepsilon=0.01, n=10000$ for $1000$ iterated samples of $\sqrt{n}\sqrt{I_{\sigma}^{22}(\theta_0)}(\wh{\beta}_{n,\varepsilon,2}-\beta_2)$.}
\label{fig4}
\end{figure}

\begin{figure}[htbp]
\centering
\includegraphics[width = 9cm]{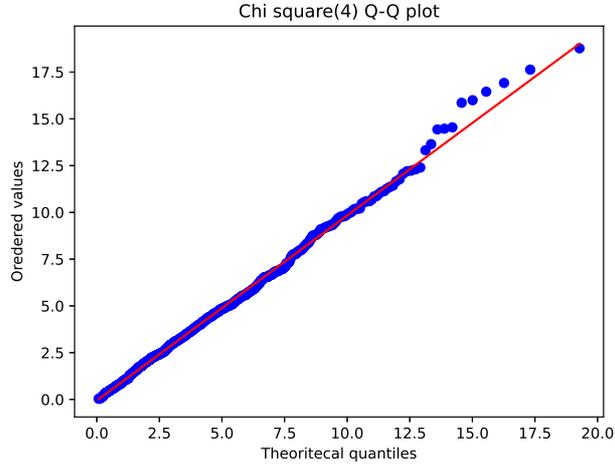}
\caption{Chi square(4) Q-Q plot as $\theta=(1.0, 2.0, 3.0, 4.0), \varepsilon=0.01, n=10000$ for $1000$ iterated samples of $\big(\varepsilon^{-1}(\wh{\alpha}_{n,\varepsilon}-\alpha), \sqrt{n}(\wh{\beta}_{n,\varepsilon,}-\beta)\big)^{\top}I(\theta_0)\big(\varepsilon^{-1}(\wh{\alpha}_{n,\varepsilon}-\alpha), \sqrt{n}(\wh{\beta}_{n,\varepsilon,}-\beta)\big)$.}
\label{fig5}
\end{figure}

\section{Proofs}\label{sec:proofs}
We first establish some preliminary lemmas. The idea of the proof of Lemma \ref{having solution} is due to that of Lemma 2.2.1 by Nualart \cite{nualart}.
\begin{lemm}\label{having solution}
Suppose that {\upshape(A\ref{assu1})} and {\upshape(A\ref{assu2})} hold true. Then there exists a strong solution $\left\{X_t^\varepsilon\right\}$ for $\varepsilon\in(0,1]$. Moreover, for $p\geq2$, it holds true:
\begin{equation*}
E\left[\sup_{-\delta\leq t\leq 1}\left|X_t^\varepsilon\right|^p\right]<\infty.
\end{equation*}

\begin{proof}[\bf{Proof}]
Let 
\begin{align*}
X_t^0=
\begin{cases}
x_0,\quad t\in[0,1];\\
\phi^\varepsilon(t),\quad t\in[-\delta,0);
\end{cases}
\end{align*}
and for $n\geq0$,
\begin{equation*}
X_t^{n+1}=
\begin{cases}
x_0^\varepsilon+\int_{0}^{t}b\left(X_s^n,H(X_{s-\cdot}^n),\theta_0\right)\,\df s+\varepsilon\int_0^t\sigma\left(X_s^n,H(X_{s-\cdot}^n),\beta_0\right)\,\df W_s,\quad t\in[0,1];\\
\phi^\varepsilon(t),\quad t\in[-\delta,0].
\end{cases}
\end{equation*}
First, we show that there exists a strong solution $\left\{X_t^\varepsilon\right\}$. From Lemma 2.2.1 of Nualart \cite{nualart}, it suffices to show that
\begin{equation}
E\left[\sup_{-\delta\leq t\leq 1}\left|X_t^{n}\right|^p\right]<\infty, \label{recursive}
\end{equation}
and
\begin{equation}
\sum_{n=0}^{\infty}E\left[\sup_{-\delta\leq t\leq 1}\left|X_{t}^{n+1}-X_{t}^{n}\right|^p\right]<\infty, \label{borel-cantelli}
\end{equation}
for any $p\geq 2$. By a recursive argument, we can show that the inequality \eqref{recursive} holds. By using the Burkholder-Davis-Gundy inequality and (A\ref{assu1}),
\begin{align*}
E\left[\sup_{-\delta\leq t\leq 1}\left|X_t^{n+1}\right|^p\right] &\leq E\left[\sup_{-\delta\leq t \leq0}\left|X_t^{n+1}\right|^p\right]+E\left[\sup_{0\leq t\leq 1}\left|X_t^{n+1}\right|^p\right]\\
&\leq C_p\Bigg\{E\left[\sup_{-\delta\leq t\leq 0}\left|\phi^\varepsilon(t)\right|^p\right] +E\left[\left|x_0^\varepsilon\right|^p\right]+E\bigg[\int_0^1\Big|b\left(X_{s}^n,H(X_{s-\cdot}^n),\theta_0\right)\Big|^p\,\df s\bigg]\\
&\quad+\varepsilon^pE\Bigg[\bigg|\int_0^1\sigma\left(X_{s}^n,H(X_{s-\cdot}^n),\theta_0\right)\,\df W_s\bigg|^p\Bigg]\Bigg\}\\
&\leq C_p\Bigg\{E\left[\sup_{-\delta\leq t\leq 0}\left|\phi^\varepsilon(t)\right|^p\right] +E\left[\big|x_0^\varepsilon\big|^p\right]\\
&\quad+K^pC'_p\int_0^1\left(1+E\big[|X_s^n|^p\big]+\left(\mu\left(\left[0,\delta\right]\right)\right)^pE\left[\sup_{0\leq u\leq \delta}\left|X_{s-u}^n\right|^p\right]\right)\,\df s\Bigg\}\\
&\leq C_p\Bigg\{E\left[\sup_{-\delta\leq t\leq 0}\left|\phi^\varepsilon(t)\right|^p\right] +E\left[\big|x_0^\varepsilon\big|^p\right]\\
&\quad+K^pC'_p\left(1+\sup_{0\leq s\leq 1}E\left[|X_s^n|^p\right]+\left(\mu\left(\left[0,\delta\right]\right)\right)^p\sup_{0\leq s\leq 1}E\left[\sup_{0\leq u\leq \delta}\left|X_{s-u}^n\right|^p\right]\right)\Bigg\}\\
&\leq C_p\Bigg\{E\left[\sup_{-\delta\leq t\leq 0}\left|\phi^\varepsilon(t)\right|^p\right] +E\left[\big|x_0^\varepsilon\big|^p\right]\\
&\quad+K^pC'_p\left(1+\sup_{0\leq s\leq 1}E\left[\left|X_s^n\right|^p\right]+\left(\mu\left(\left[0,\delta\right]\right)\right)^pE\left[\sup_{-\delta\leq s\leq 1}\left|X_{s}^n\right|^p\right]\right)\Bigg\},
\end{align*}
where $C_p$ and $C'_p$ are constants depending only on $p$. For \eqref{borel-cantelli}, by applying the Burkholder-Davis-Gundy inequality and (A\ref{assu1}) again, we have
\begin{align*}
E\left[\sup_{-\delta \leq t\leq 1}\left|X_{t}^{n+1}-X_{t}^{n}\right|^p\right]&\leq C_p\left\{1+\left(\mu\left(\left[0,\delta\right]\right)\right)^p\right\}\int_0^1E\left[\sup_{0\leq u\leq\delta}\left|X_{s-u}^n-X_{s-u}^{n-1}\right|^p\right]\,\df s\\
&\leq \frac{1}{n!}\left[C_p\left\{1+\left(\mu\left(\left[0,\delta\right]\right)\right)^p\right\}\right]^{n+1}E\left[\sup_{-\delta\leq t\leq 1}|X_t^1|^p\right].
\end{align*}
Consequently, we have the inequality \eqref{borel-cantelli} by \eqref{recursive}. 

Finally, we shall prove that the solution of \eqref{SFDEs} is unique. We assume that $\tilde{X}_t^\varepsilon$ is the solution of \eqref{SFDEs}. 
Then, it follows by (A\ref{assu1}) that
\begin{multline*}
E\left[\sup_{-\delta\leq u \leq 0}\left|X_{t-u}^\varepsilon-\tilde{X}_{t-u}^\varepsilon\right|^2\right]\leq 8K^2\Bigg\{\int_0^tE\left[\sup_{-\delta\leq u \leq 0}\left|X_{s-u}^\varepsilon-\tilde{X}_{s-u}^\varepsilon\right|^2\right]\,\df s\\
+\mu\left(\left[0,\delta\right]\right)\int_0^tE\left[\sup_{-\delta\leq u \leq 0}\left|X_{s-u}^\varepsilon-\tilde{X}_{s-u}^\varepsilon\right|^2\right]\,\df s\Bigg\}.
\end{multline*}
Hence, it follows from Gronwall's inequality that
\begin{equation*}
E\left[\sup_{-\delta\leq u \leq 0}\left|X_{t-u}^\varepsilon-\tilde{X}_{t-u}^\varepsilon\right|^2\right]=0.
\end{equation*}
The proof is completed.
\end{proof}
\end{lemm}
\noindent
For Lemma \ref{X_t go to ODE}, we shall use the notations:
\begin{enumerate}
\renewcommand{\labelenumi}{(N\theenumi)}
\setcounter{enumi}{\value{enumimemory}}
\item Denote by $Y_t^{n,\varepsilon}:=X_{\lfloor nt \rfloor/t}^\varepsilon$ and $Y_{t-\cdot}^{n,\varepsilon}:=X_{\lfloor nt \rfloor/t-\cdot}^\varepsilon$ 
for the stochastic process $X^\varepsilon$ defined by \eqref{SFDEs}.
\item For $X_{t-\cdot}^\varepsilon\in C\left([0,\delta];\mathbb{R}^d\right)$, denote by $\left\|X_{t-\cdot}^\varepsilon\right\|_\infty:=\sup_{0\leq u \leq \delta}\left|X_{t-u}^\varepsilon\right|.$

\setcounter{enumimemory}{\value{enumi}}
\end{enumerate}
In Lemma \ref{X_t go to ODE}, the proof ideas follow Long et al. \cite{long}
\begin{lemm} \label{X_t go to ODE}
Suppose that {\upshape(A\ref{assu1})} and {\upshape(A\ref{assu2})} hold true. Then, it follows for $p\geq1$, 
\begin{align*}
E\left[\sup_{0\leq t\leq 1}\left|Y_{t}^{n,\varepsilon}-X_{t}^{0}\right|^p\right]&=O\left(\varepsilon^p\right)+O\left(1/n^p\right); \\
E\left[\sup_{0\leq t\leq 1}\left\|Y_{t-\cdot}^{n,\varepsilon}-X_{t-\cdot}^{0}\right\|_{\infty}^p\right]&=O\left(\varepsilon^p\right)+O\left(1/n^p\right)
\end{align*} 
as $\varepsilon\to 0$ and $n\rightarrow\infty$.
\end{lemm}
\begin{proof}[\bf{Proof}]
Since it is easy to see that 
\begin{equation*}
\left|Y_{t}^{n,\varepsilon}-X_t^0\right|^p\leq\left\|Y_{t-\cdot}^{n,\varepsilon}-X_{t-\cdot}^0\right\|_{\infty}^p\quad a.s.,
\end{equation*}
it suffices to show that $E\left[\sup_{0\leq t\leq 1}\left\|Y_{t-\cdot}^{n,\varepsilon}-X_{t-\cdot}^{0}\right\|_{\infty}^p\right]=O\left(\varepsilon^p\right)+O\left(1/n^p\right)$ ~ as $\varepsilon\rightarrow 0$ and $n\to\infty$. 
We shall prove only the case where $\delta \ge 1$ because the proof for $\delta\leq 1$ is almost the same. It follows from {\upshape(A\ref{assu1})} and the Burkholder-Davis-Gundy inequality that
\begin{align*}
E\left[\sup_{0\leq t\leq 1}\left\|X_{t-\cdot}^\varepsilon-X_{t-\cdot}^{0}\right\|_{\infty}^p\right] &\leq E\left[\sup_{0\leq t\leq 1}\sup_{0\leq s\leq t}\big\|X_{t-\cdot}^\varepsilon-X_{t-\cdot}^{0}\big\|_{\infty}^p\right]+E\left[\sup_{0\leq t\leq 1}\sup_{t\leq s\leq \delta}\big\|X_{t-\cdot}^\varepsilon-X_{t-\cdot}^{0}\big\|_\infty^p\right]\\
&\leq 2^pE\Biggl[\sup_{0\leq t \leq 1}\bigg|\int_{0}^{t}b\big(X_{s}^\varepsilon,H(X_{s-\cdot}^\varepsilon),\theta_{0}\big)-b\big(X_{s}^{0},H(X_{s-\cdot}^0),\theta_{0}\big)\,\df s\bigg|^p\Biggr]\\
&\quad +2^p\varepsilon^p E\Biggl[\sup_{0\leq t \leq 1}\bigg|\int_{0}^{t}\sigma\big(X_{s}^\varepsilon,H(X_{s-\cdot}^\varepsilon),\beta_{0}\big)\,\df W_s\bigg|^p\Biggr]\\
&\quad +E\left[\sup_{0\leq t\leq 1}\sup_{t\leq s\leq \delta}\bigl|\phi^\varepsilon(t-s)-\phi(t-s)\bigr|^p\right]\\
&\leq 2^pK^pE\left[\sup_{0\leq t \leq 1}\int_{0}^{t}\big|X_{s}^\varepsilon-X_{s}^{0}\big|^p+\big|H(X_{s-\cdot}^\varepsilon)-H(X_{s-\cdot}^0)\big|^p\,\df s\right]\\
&\quad +\varepsilon^p C_pE\left[\int_0^1\Big|\sigma_{ij}\big(X_{s}^\varepsilon,H(X_{s-\cdot}^\varepsilon),\beta_{0}\big)\Big|^{p/2}\,\df s\right]\\
&\quad +E\left[\sup_{-\delta \leq t\leq 0}\bigl|\phi^\varepsilon(t)-\phi(t)\bigr|^p\right]\\
&\leq 2^pK^p\left(1+\mu\left([0,\delta]\right)^p\right)E\left[\int_{0}^{1}\sup_{0\leq v\leq s}\|X_{v-\cdot}^\varepsilon-X_{v-\cdot}^0\big\|_\infty^p\,\df s\right]\\
&\quad + \varepsilon^p C_p\Big(1+E\left[\sup_{0\leq s \leq 1}\big|X_{s}\big|^{p/2}\right]+\mu([0,\delta])^{p/2}E\left[\sup_{0\leq s \leq 1}\big\|X_{s-\cdot}^0\big\|_{\infty}^{p/2}\right]\Big)\\
&\quad +E\left[\sup_{-\delta \leq t\leq 0}\bigl|\phi^\varepsilon(t)-\phi(t)\bigr|^p\right],
\end{align*}
where $C_p$ is a constant depending only on $p$. It holds from Gronwall's inequality and (A\ref{assu2}) that
\begin{align*}
E\left[\sup_{0\leq t\leq 1}\left\|X_{t-\cdot}^\varepsilon-X_{t-\cdot}^{0}\right\|_{\infty}^p\right]&\leq \Biggl\{ \varepsilon^p C_p\bigg(1+E\left[\sup_{0\leq s \leq 1}\big|X_{s}\big|^{p/2}\right]+\mu([0,\delta])^{p/2}E\left[\sup_{0\leq s \leq 1}\big\|X_{s-\cdot}^0\big\|_{\infty}^{p/2}\right]\bigg)\\
&\quad +E\left[\sup_{-\delta \leq t\leq 0}\bigl|\phi^\varepsilon(t)-\phi(t)\bigr|^p\right]\Biggr\} e^{2^pK^p\left(1+\mu\left([0,\delta]\right)^p\right)}\\
&=O(\varepsilon^p),
\end{align*}
as $\varepsilon\to 0$. From the continuity of $X_t^0$, the proof is completed.
\end{proof}
\begin{remark}
\upshape
It is satisfied from the proof of Lemma \ref{X_t go to ODE} and {\upshape(A\ref{assu2})} that $\sup_{\varepsilon\in(0,1]}E\left[\sup_{-\delta \leq t\leq 1}\left|X_t^\varepsilon\right|^p\right]<\infty$ for $p\geq 1$.
\end{remark}
\begin{lemm}\label{H(Y) to H(x) and H_n(Y) to H(x)}
Suppose the conditions {\upshape(A\ref{assu1})} and {\upshape(A\ref{assu2})}. Then the following (i) and (ii) hold true as $\varepsilon\to 0$ and $n\to \infty$: for any $p\geq1$,
\renewcommand{\labelenumi}{\rm{(\roman{enumi})}}
\begin{enumerate}
\item 
$\displaystyle 
E\left[\sup_{t\in[0,1]}\left|H\left(Y_{t-\cdot}^{n,\varepsilon}\right)-H\left(X_{t-\cdot}^0\right)\right|^p\right]=O\left(\varepsilon^p\right)+O\left(1/n^p\right). 
$

\item
$\displaystyle 
E\left[\sup_{t\in[0,1]}\left|H_n\left(Y_{t-\cdot}^{n,\varepsilon}\right)-H\left(X_{t-\cdot}^0\right)\right|^p\right]=O\left(\varepsilon^p\right)+O\left(1/n^p\right).$ 
\end{enumerate}
\begin{proof}[\bf{Proof}]
(i) From Lemma \ref{X_t go to ODE},
\begin{align*}
E\left[\sup_{t\in[0,1]}\left|H\left(Y_{t-\cdot}^{n,\varepsilon}\right)-H\left(X_{t-\cdot}^0\right)\right|^p\right] 
&\leq \mu\left([0,\delta]\right) E\left[\sup_{t\in[0,1]}\left\|Y_{t-\cdot}^{n,\varepsilon}-X_{t-\cdot}^{0}\right\|_{\infty}^p\right]\\
&= O\left(\varepsilon^p\right)+O\left(1/n^p\right),
\end{align*}
as $\varepsilon\to 0$ and $n\to \infty$.\\
(ii) From Lemma \ref{X_t go to ODE} and the continuity of $X^0_t$, we have
\begin{align*}
E\left[\sup_{t\in[0,1]}\left|H_n\left(Y_{t-\cdot}^{n,\epsilon}\right)-H\left(X_{t-\cdot}^0\right)\right|^{p}\right]&\leq \int_{0}^{\delta}E\left[\sup_{t\in[0,1]}\left|Y_{t-\lfloor ns \rfloor/n}^{n,\epsilon}-X_{t-s}^0\right|^{p}\right]\,\mu(\df s)\\
&\leq 2^{p-1}\int_{0}^{\delta}\Bigg\{E\left[\sup_{t\in[0,1]}\left|Y_{t-\lfloor ns \rfloor/n}^{n,\epsilon}-X_{\lfloor nt \rfloor/n-\lfloor ns \rfloor/n}^0\right|^p\right]\\
&\quad +E\left[\sup_{t\in[0,1]}\left|X_{\lfloor nt \rfloor/n-\lfloor ns \rfloor/n}^0-X_{t-s}^0\right|^p\right]\Bigg\}\,\mu(\df s)\\
&= O\left(\epsilon^p\right)+O\left(1/n^p\right),
\end{align*}
as $\varepsilon\to 0$ and $n\to \infty$.
\end{proof}
\end{lemm}
\begin{lemm}\label{order of SFDEs}
Suppose the conditions {\upshape(A\ref{assu1})} and {\upshape(A\ref{assu2})}. Then it holds that
\begin{align*}
E\left[\left|H(X_{t-\cdot})-H_n(X_{t-\cdot})\right|^p\right]=O\left(n^{-p}\right)+O\left(\varepsilon^pn^{-p/2}\right),
\end{align*}
for $p\geq1$.
\begin{proof}[\bf{Proof}]
From {\upshape(A\ref{assu1})} and {\upshape(A\ref{assu2})}, we find that
\begin{align*}
E\left[\left|H\left(X_{t-\cdot}\right)-H_n\left(X_{t-\cdot}\right)\right|^p\right] &\leq \int_{0}^{\delta}E\left[\left|X_{t-s}-X_{t-\lfloor ns \rfloor/n}\right|^p\right]\,\mu(\df s)\\
&=O\left(n^{-p}\right)+O\left(\varepsilon^pn^{-p/2}\right).
\end{align*}
\end{proof}
\end{lemm}
\noindent
For the following Lemmas, we shall use the notations:
\begin{enumerate}
\renewcommand{\labelenumi}{(N\theenumi)}
\setcounter{enumi}{\value{enumimemory}}
\item let $R_{k-1}$ denote a function $(0,1]\to\mathbb{R}$ for which there exists a constant $C$ such that 
\begin{equation*}
\left|R_{k-1}(a)\right|\leq aC\left(1+\|X_{t_{k-1}-\cdot}\|_{\infty}\right)^C,
\end{equation*}
for all $a > 0$.
\setcounter{enumimemory}{\value{enumi}}
\end{enumerate}
\begin{lemm} \label{prepare for conditional expectation}
Suppose the conditions {\upshape(A\ref{assu1})} and {\upshape(A\ref{assu2})}. For $p\geq1$ and $t_{k-1}\leq t\leq t_k$, it holds that
\begin{align*}
E\left[\left|X_t^\varepsilon-X_{t_{k-1}}^\varepsilon\right|^p+\left|H\left(X_{t-\cdot}^\varepsilon\right)-H\bigl(X_{t_{k-1}-\cdot}^\varepsilon\bigr)\right|^p\bigg|\mathcal{F}_{t_{k-1}}\right]&\leq \Phi_p^\varepsilon(t)+R_{k-1}\left(\left(t-t_{k-1}\right)^p\right)\\
&\quad +R_{k-1}\left(\varepsilon^p\left(t-t_{k-1}\right)^{p/2}\right)
\end{align*}
where $C_p$ is constant depending only on $p$, and $\Phi_p^\varepsilon(\cdot)$ is a function: 
\begin{equation*}
\Phi_p^\varepsilon(t)=C_p\left\{\int_{[t_{k-1},t_k]}\left|\phi^\varepsilon(0)-\phi^\varepsilon(t_{k-1}-s)\right|^p\,\mu(\df s)+\int_{[0,\delta]}\left|\phi^\varepsilon(t-s)-\phi^\varepsilon(t_{k-1}-s)\right|^p\,\mu(\df s)\right\}.
\end{equation*}
\end{lemm}
\begin{proof}[\bf{Proof}]
In the same way as Lemma 6 in Kessler \cite{kessler}, we have
\begin{align}
E\left[\left|X_{t}^\varepsilon-X_{t_{k-1}}^\varepsilon\right|^{p}\bigg|\mathcal{F}_{t_{k-1}}\right]&\leq C_p \int_{t_{k-1}}^{t}E\left[\left|X_{s}^\varepsilon-X_{t_{k-1}}^\varepsilon\right|^{p}+\left|H\left(X_{s-\cdot}^\varepsilon\right)-H\big(X_{t_{k-1}-\cdot}^\varepsilon\big)\right|^{p}\bigg|\mathcal{F}_{t_{k-1}}\right]\,\df s\notag\\
&\quad +R_{k-1}\Big(\big(t-t_{k-1}\big)^p\Big)+R_{k-1}\Big(\varepsilon^p\big(t-t_{k-1}\big)^{p/2}\Big), \label{same kessler}
\end{align}
where $C_p$ is constant depending only on $p$. We find that
\begin{align*}
E\left[\left|H\left(X_{t-\cdot}^\varepsilon\right)-H\big(X_{t_{k-1}-\cdot}^\varepsilon\big)\right|^p\bigg|\mathcal{F}_{t_{k-1}}\right]\leq \int_0^\delta E\left[\left|X_{t-s}^\varepsilon-X_{t_{k-1}-s}^\varepsilon\right|^p\bigg|\mathcal{F}_{t_{k-1}}\right]\,\mu(\df s).
\end{align*}
Next, we consider three cases:
\begin{center}
(b1)\ \ $t_{k-1}> s$;\qquad (b2)\ \ $t_{k-1}\leq s < t$;\qquad (b3)\ \ $t\leq s$.
\end{center}

\noindent (b1) In the same way as \eqref{same kessler}, we find that
\begin{align}
E\left[\left|X_{t-s}^\varepsilon-X_{t_{k-1}-s}^\varepsilon\right|^p\bigg|\mathcal{F}_{t_{k-1}}\right]&\leq C_p\int_{t_{k-1}-s}^{t-s}E\left[\left|X_u^\varepsilon-X_{t_{k-1}}^\varepsilon\right|^{p}|\mathcal{F}_{t_{k-1}}\right]\,\df u\notag\\
&\quad +C_p\int_{t_{k-1}-s}^{t-s}E\left[\left|H\left(X_{u-\cdot}^\varepsilon\right)-H\big(X_{t_{k-1}-\cdot}^\varepsilon\big)\right|^{p}\bigg|\mathcal{F}_{t_{k-1}}\right]\,\df u\notag\\
&\quad +R_{k-1}\Big(\big(t-t_{k-1}\big)^p\Big)+R_{k-1}\Big(\varepsilon^p\big(t-t_{k-1}\big)^{p/2}\Big).\label{b1}
\end{align}

\noindent (b2) In the same way as \eqref{same kessler},
\begin{align}
E\left[\left|X_{t-s}^\varepsilon-X_{t_{k-1}-s}^\varepsilon\right|^p\bigg|\mathcal{F}_{t_{k-1}}\right]&= E\left[\left|X_{t-s}^\varepsilon-\phi(t_{k-1}-s)\right|^p\bigg|\mathcal{F}_{t_{k-1}}\right]\notag\\
&\leq2^{p-1}\left(E\left[\left|X_{t-s}^\varepsilon-X_0^\varepsilon\right|^p\bigg|\mathcal{F}_{t_{k-1}}\right]+\left|\phi(0)-\phi(t_{k-1}-s)\right|^p\right)\notag\\
&\leq \tilde{C}_p\int_{0}^{t-s}E\left[\left|X_{u}^\varepsilon-X_{t_{k-1}}^\varepsilon\right|^{p}+\left|H\left(X_{u-\cdot}^\varepsilon\right)-H\big(X_{t_{k-1}-\cdot}^\varepsilon\big)\right|^{p}\bigg|\mathcal{F}_{t_{k-1}}\right]\,\df u \notag\\
&\quad +2^{p-1}|\phi(0)-\phi(t_{k-1}-s)|^p \notag\\
&\quad +R_{k-1}\Big(\big(t-t_{k-1}\big)^p\Big)+R_{k-1}\Big(\varepsilon^p\big(t-t_{k-1}\big)^{p/2}\Big), \label{b2}
\end{align}
\quad where $\tilde{C}_p$ is constant depending only on $p$.\\

\noindent (b3) It is easy to find that
\begin{equation}
E\left[\left|X_{t-s}^\varepsilon-X_{t_{k-1}-s}^\varepsilon\right|^p\bigg|\mathcal{F}_{t_{k-1}}\right]=\left|\phi(t-s)-\phi(t_{k-1}-s)\right|^p.\label{b3}
\end{equation}
From \eqref{b1}, \eqref{b2} and \eqref{b3}, we have
\begin{align*}
E\left[\left|H\left(X_{t-\cdot}^\varepsilon\right)-H\big(X_{t_{k-1}-\cdot}^\varepsilon\big)\right|^p\bigg|\mathcal{F}_{t_{k-1}}\right]&\leq C_p\bigg\{\mu\left(\left[0,\delta\right]\right) \int_{0}^{t}E\left[\left|X_{u}^\varepsilon-X_{t_{k-1}}^\varepsilon\right|^{p}\bigg|\mathcal{F}_{t_{k-1}}\right]\,\df u\\
&\qquad +\mu\left(\left[0,\delta\right]\right) \int_{0}^{t}E\left[\left|H\left(X_{u-\cdot}^\varepsilon\right)-H\big(X_{t_{k-1}-\cdot}^\varepsilon\big)\right|^{p}\bigg|\mathcal{F}_{t_{k-1}}\right]\,\df u\\
&\qquad +\int_{[t_{k-1},t_k]}\left|\phi^\varepsilon(0)-\phi^\varepsilon(t_{k-1}-s)\right|^p\,\mu\left(\df s\right)\\
&\qquad +\int_{[0,\delta]}\left|\phi^\varepsilon(t-s)-\phi^\varepsilon(t_{k-1}-s)\right|^p\,\mu\left(\df s\right)\bigg\}\\
&\quad +R_{k-1}\Big(\big(t-t_{k-1}\big)^p\Big)+R_{k-1}\Big(\varepsilon^p\big(t-t_{k-1}\big)^{p/2}\Big).
\end{align*}
By {\upshape(A\ref{assu1})}, \eqref{same kessler} and Gronwall's inequality, we obtain the conclusion.
\end{proof}
\begin{remark}
\upshape
Lemma \ref{prepare for conditional expectation} is satisfied for $t_{k-1}\geq \delta$ and for $t\geq \delta$.
\end{remark}
\noindent
For Lemma \ref{conditional expectation}, we use some abbreviations:
\begin{enumerate}
\renewcommand{\labelenumi}{(N\theenumi)}
\setcounter{enumi}{\value{enumimemory}}
\item Denote by
$b^i_{t,H}=b^i\left(X_t^\varepsilon,H\left(X_{t-\cdot}^\varepsilon\right),\theta_0\right)$ and
\begin{equation*}
[\sigma\sigma^\top]_{t,H}^{ij}=[\sigma\sigma^\top]^{ij}\left(X_t^\varepsilon,H\left(X_{t-\cdot}^\varepsilon\right),\beta_0\right)
\end{equation*}
for $\left\{b^i\right\}$ and $\left\{[\sigma\sigma^{\top}]^{ij}\right\}$.
\setcounter{enumimemory}{\value{enumi}}
\end{enumerate}
\begin{lemm} \label{conditional expectation}
Suppose the condition {\upshape(A\ref{assu1})}. Then the following (i)-(iv) hold true:
\renewcommand{\labelenumi}{\rm{(\roman{enumi})}}
\begin{enumerate}
\item \label{1dim} 
\begin{align*}
E\left[P_{k}^{i_{1}}(\theta_{0})\Big|\mathcal{F}_{t_{k-1}}\right]=\frac{1}{n}\left(b^{i_{1}}_{t_{k-1},H}-b^{i_{1}}_{t_{k-1},H_n}\right)+\int_{t_{k-1}}^{t_k}\Phi_1^\varepsilon(s)~ds+R_{k-1}\left(\frac{1}{n^2}\right)+R_{k-1}\left(\frac{\varepsilon}{n\sqrt{n}}\right). 
\end{align*}
\item \label{2dim}  
\begin{align*}
E\left[P_{k}^{i_{1}}P_{k}^{i_{2}}\left(\theta_{0}\right)\Big|\mathcal{F}_{t_{k-1}}\right]&= \frac{\varepsilon^{2}}{n}[\sigma\sigma^\top]_{t_{k-1},H}^{i_{1}i_{2}}+\frac{1}{n^2}\left(b_{t_{k-1},H}^{i_1}-b_{t_{k-1},H_n}^{i_1}\right)\left(b_{t_{k-1},H}^{i_2}-b_{t_{k-1},H_n}^{i_2}\right)\\
&\quad +\int_{t_{k-1}}^{t_k}\left\{\Phi_2^\varepsilon(s)+\varepsilon^2\Phi_1^\varepsilon(s)\right\}\,\df s\\
&\quad +R_{k-1}\left(\frac{1}{n}\right)\int_{t_{k-1}}^{t_k}\Phi_1^\varepsilon(s)\,\df s+R_{k-1}(1)\int_{t_{k-1}}^{t_k}\int_{t_{k-1}}^{s}\Phi_1^\varepsilon(u)\,\df u\,\df s\\
&\quad+R_{k-1}\left(\frac{1}{n^3}\right)+R_{k-1}\left(\frac{\varepsilon}{n^2\sqrt{n}}\right)+R_{k-1}\left(\frac{\varepsilon^2}{n^2}\right)+R_{k-1}\left(\frac{\varepsilon^3}{n\sqrt{n}}\right).
\end{align*}
\item \label{3dim}  
\begin{align*}
E\left[P_{k}^{i_{1}}P_{k}^{i_{2}}P_{k}^{i_{3}}(\theta_{0})\Big|\mathcal{F}_{t_{k-1}}\right]&= \frac{1}{n^3}\left(b_{t_{k-1},H}^{i_1}-b_{t_{k-1},H_n}^{i_1}\right)\left(b_{t_{k-1},H}^{i_2}-b_{t_{k-1},H_n}^{i_2}\right)\left(b_{t_{k-1},H}^{i_3}-b_{t_{k-1},H_n}^{i_3}\right)\\
&\quad +\frac{\varepsilon^2}{2n^2}\sum_{\tilde{\sigma}\in G(3)}[\sigma\sigma^\top]^{i_{\tilde{\sigma}(1)}i_{\tilde{\sigma}(2)}}\left(b_{t_{k-1},H}^{i_{\tilde{\sigma}(3)}}-b_{t_{k-1},H_n}^{i_{\tilde{\sigma}(3)}}\right)\\
&\quad +\int_{t_{k-1}}^{t_k}\left\{\Phi_3^\varepsilon(s)+\varepsilon^2\Phi_2^\varepsilon(s)\right\}\,\df s\\
&\quad +R_{k-1}\left(\frac{1}{n}\right)\int_{t_{k-1}}^{t_k}\{\Phi_2^\varepsilon(s)+\varepsilon^2\Phi_1^\varepsilon(s)\}\,\df s\\
&\quad +R_{k-1}\left(\frac{1}{n^2}\right)b_{t_{k-1},H_n}^{i_{\tilde{\sigma}(2)}}\int_{t_{k-1}}^{t_k}\Phi_1^\varepsilon(s)\,\df s\\
&\quad +R_{k-1}(1)\int_{t_{k-1}}^{t_k}\int_{t_{k-1}}^{s}\{\Phi_2^\varepsilon(u)+\varepsilon^2\Phi_1^\varepsilon(u)\}\,\df u\,\df s\\
&\quad +R_{k-1}\left(\frac{1}{n}\right)\int_{t_{k-1}}^{t_k}\int_{t_{k-1}}^s\Phi_1^\varepsilon(u)\,\df u\,\df s\\
&\quad +\sum_{\tilde{\sigma}\in G(3)}b_{t_{k-1},H}^{i_{\tilde{\sigma}(1)}}b_{t_{k-1},H}^{i_{\tilde{\sigma}(2)}}\int_{t_{k-1}}^{t_k}\int_{t_{k-1}}^{s}\int_{t_{k-1}}^{u}\Phi_1^\varepsilon(v)\,\df v\,\df u\,\df s\\
&\quad +R_{k-1}\left(\frac{1}{n^4}\right)+R_{k-1}\left(\frac{\varepsilon}{n^3\sqrt{n}}\right)+R_{k-1}\left(\frac{\varepsilon^2}{n^3}\right)\\
&\quad +R_{k-1}\left(\frac{\varepsilon^3}{n^2\sqrt{n}}\right)+R_{k-1}\left(\frac{\varepsilon^4}{n^2}\right).
\end{align*}
\item \label{4dim}  
\begin{align*}
E\left[P_{k}^{i_{1}}P_{k}^{i_{2}}P_{k}^{i_{3}}P_{k}^{i_{4}}(\theta_{0})\Big|\mathcal{F}_{t_{k-1}}\right]&= \frac{\varepsilon^4}{n^2}\Bigg([\sigma\sigma^\top]_{t_{k-1},H}^{i_1i_2}[\sigma\sigma^\top]_{t_{k-1},H}^{i_3i_4}+[\sigma\sigma^\top]_{t_{k-1},H}^{i_1i_3}[\sigma\sigma^\top]_{t_{k-1},H}^{i_2i_4}\\
&\qquad +[\sigma\sigma^\top]_{t_{k-1},H}^{i_1i_4}[\sigma\sigma^\top]_{t_{k-1},H}^{i_2i_3}\Bigg)\\ 
&\quad +\frac{1}{n^4}\left(b_{t_{k-1},H}^{i_1}-b_{t_{k-1},H_n}^{i_1}\right)\left(b_{t_{k-1},H}^{i_2}-b_{t_{k-1},H_n}^{i_2}\right)\times\\
&\qquad \left(b_{t_{k-1},H}^{i_3}-b_{t_{k-1},H_n}^{i_3}\right)\left(b_{t_{k-1},H}^{i_4}-b_{t_{k-1},H_n}^{i_4}\right)\\
&\quad +\frac{\varepsilon^2}{4n^3}\sum_{\tilde{\sigma}\in G(4)}[\sigma\sigma^{\top}]^{i_{\tilde{\sigma}(1)}i_{\tilde{\sigma}(2)}}\left(b_{t_{k-1},H}^{i_{\tilde{\sigma}(3)}}-b_{t_{k-1},H_n}^{i_{\tilde{\sigma}(3)}}\right)\left(b_{t_{k-1},H}^{i_{\tilde{\sigma}(4)}}-b_{t_{k-1},H_n}^{i_{\tilde{\sigma}(4)}}\right)\\
&\quad +\int_{t_{k-1}}^{t_k}\left\{\Phi_4^\varepsilon(s)+\varepsilon^2\Phi_3^\varepsilon(s)\right\}\,\df s\\
&\quad +R_{k-1}\left(\frac{1}{n}\right)\int_{t_{k-1}}^{t_k}\{\Phi_3^\varepsilon(u)+\varepsilon^2\Phi_2^\varepsilon(s)\}\,\df s\\
&\quad +R_{k-1}\left(\frac{1}{n^2}\right)\int_{t_{k-1}}^{t_k}\{\Phi_2^\varepsilon(s)+\varepsilon^2\Phi_1^\varepsilon(s)\}\,\df s\\
&\quad +R_{k-1}(1)\int_{t_{k-1}}^{s}\{\Phi_3^\varepsilon(u)+\varepsilon^2\Phi_2^\varepsilon(u)\}\,\df u\,\df s\\
&\quad +R_{k-1}\left(\frac{1}{n}\right)\int_{t_{k-1}}^{t_k}\int_{t_{k-1}}^{s}\{\Phi_2^\varepsilon(u)+\varepsilon^2\Phi_1^\varepsilon(u)\}\,\df u\,\df s\\
&\quad +R_{k-1}(\varepsilon^2)\int_{t_{k-1}}^{t_k}\int_{t_{k-1}}^{s}\{\Phi_2^\varepsilon(u)+\varepsilon^2\Phi_1^\varepsilon(u)\}\,\df u\,\df s\\
&\quad +R_{k-1}(1)\int_{t_{k-1}}^{t_k}\int_{t_{k-1}}^{s}\int_{t_{k-1}}^{u}\left\{\Phi_2^\varepsilon(v)+\varepsilon^2\Phi_1^\varepsilon(v)\right\}\,\df v\,\df u\,\df s\\
&\quad +R_{k-1}(1)\int_{t_{k-1}}^{t_k}\int_{t_{k-1}}^s\int_{t_{k-1}}^u\int_{t_{k-1}}^v\Phi_1^\varepsilon(w)\,\df w\,\df v\,\df u\,\df s\\
&\quad +R_{k-1}\left(\frac{1}{n^5}\right)+R_{k-1}\left(\frac{\varepsilon}{n^4\sqrt{n}}\right)+R_{k-1}\left(\frac{\varepsilon^2}{n^4}\right)\\
&\quad +R_{k-1}\left(\frac{\varepsilon^3}{n^3\sqrt{n}}\right)+R_{k-1}\left(\frac{\varepsilon^4}{n^3}\right)+R_{k-1}\left(\frac{\varepsilon^5}{n^2\sqrt{n}}\right).
\end{align*}
\end{enumerate}
\begin{proof}[\bf{Proof}]
(i) It is satisfied from the Lipschitz condition on $b$ in (A\ref{assu1}) that
\begin{align*}
E\left[\left(\Delta_kX^\varepsilon\right)^{i_{1}}-\frac{1}{n}b_{t_{k-1},H}^{i_{1}}\bigg|\mathcal{F}_{t_{k-1}}\right]&= E\left[\int_{t_{k-1}}^{t_{k}}\left(b_{s,H}^{i_{1}}-b_{t_{k-1},H}^{i_{1}}\right)\,\df s \bigg|\mathcal{F}_{t_{k-1}}\right]\\
&\leq E\left[\int_{t_{k-1}}^{t_{k}}\left|b_{s,H}^{i_{1}}-b_{t_{k-1},H}^{i_{1}}\right|\,\df s \bigg|\mathcal{F}_{t_{k-1}}\right]\\
&\leq E\left[\int_{t_{k-1}}^{t_{k}}\left\{\left|X_{s}^\varepsilon-X_{t_{k-1}}^\varepsilon\right|+\left|H\left(X_{s-\cdot}^\varepsilon\right)-H\big(X_{t_{k-1}-\cdot}^\varepsilon\big)\right|\right\}\,\df s \bigg|\mathcal{F}_{t_{k-1}} \right].\\
\end{align*}
From Lemma \ref{prepare for conditional expectation}, the proof of (i) is completed. \\
(ii) We find that
\begin{align*}
E\left[\prod_{j=1}^2\left(\Delta_kX^\varepsilon\right)^{i_j}\Bigg|\mathcal{F}_{t_{k-1}}\right]&= E\left[\int_{t_{k-1}}^{t_{k}}\left\{\sum_{\tilde{\sigma}\in G(2)}\left(\Delta_sX^\varepsilon\right)^{i_{\tilde{\sigma}(1)}}b_{s,H}^{i_{\tilde{\sigma}(2)}}+\varepsilon^{2}[\sigma\sigma^\top]_{s,H}^{i_{1}i_{2}}\right\}\,\df s\Bigg|\mathcal{F}_{t_{k-1}}\right]\\
&= \sum_{\tilde{\sigma}\in G(2)}E\left[\int_{t_{k-1}}^{t_{k}}\left(\Delta_sX^\varepsilon\right)^{i_{\tilde{\sigma}(1)}}\left(b_{s,H}-b_{t_{k-1},H}\right)^{i_{\tilde{\sigma}(2)}}\,\df s\Bigg|\mathcal{F}_{t_{k-1}}\right]\\
&\quad +\sum_{\tilde{\sigma}\in G(2)}E\left[\int_{t_{k-1}}^{t_{k}}\left(\Delta_sX^\varepsilon\right)^{i_{\tilde{\sigma}(1)}}\,\df s\Bigg|\mathcal{F}_{t_{k-1}}\right]b_{t_{k-1},H}^{i_{\tilde{\sigma}(2)}}\\
&\quad +\varepsilon^{2}E\left[\int_{t_{k-1}}^{t_{k}}\left\{[\sigma\sigma^\top]_{s,H}^{i_{1}i_{2}}-[\sigma\sigma^\top]_{t_{k-1},H}^{i_{1}i_{2}}\right\}\,\df s\Bigg|\mathcal{F}_{t_{k-1}}\right]\\
&\quad +\varepsilon^{2}n^{-1}[\sigma\sigma^\top]_{t_{k-1},H}^{i_{1}i_{2}}.
\end{align*}
It follows from the Lipschitz condition on $b$ in (A\ref{assu1}) that
\begin{align*}
E\left[\int_{t_{k-1}}^{t_{k}}\left(\Delta_sX^\varepsilon\right)^{i_{1}}\{b_{s,H}^{i_{2}}-b_{t_{k-1},H}^{i_{2}}\}\,\df s\bigg|\mathcal{F}_{t_{k-1}}\right]&\leq
2K\int_{t_{k-1}}^{t_{k}}E\left[\left|X_{s}^\varepsilon-X_{t_{k-1}}^\varepsilon\right|^{2}\bigg|\mathcal{F}_{t_{k-1}}\right]\,\df s\\
&\quad +2K\int_{t_{k-1}}^{t_{k}}E\left[\left|H\left(X_{s-\cdot}^\varepsilon\right)-H\big(X_{t_{k-1}-\cdot}^\varepsilon\big)\right|^{2}\bigg|\mathcal{F}_{t_{k-1}}\right]\,\df s.
\end{align*}
From Lemma \ref{prepare for conditional expectation}, we have
\begin{align}
E\left[\int_{t_{k-1}}^{t_{k}}\left(\Delta_sX^\varepsilon\right)^{i_{1}}\left\{b_{s,H}^{i_{2}}-b_{t_{k-1},H}^{i_{2}}\right\}\,\df s\bigg|\mathcal{F}_{t_{k-1}}\right]&= \int_{t_{k-1}}^{t_k}\Phi_2^\varepsilon(s)\, \df s\notag\\
&\quad +R_{k-1}\left(\frac{1}{n^3}\right)+R_{k-1}\left(\frac{\varepsilon^2}{n^2}\right).\label{R1}
\end{align}
From the same argument of \eqref{R1}, it holds that
\begin{align}
\varepsilon^{2}E\left[\int_{t_{k-1}}^{t_{k}}\left\{[\sigma\sigma^\top]_{s,H}^{i_{1}i_{2}}-[\sigma\sigma^\top]_{t_{k-1},H}^{i_{1}i_{2}}\right\}\,\df s\bigg|\mathcal{F}_{t_{k-1}}\right]&= \varepsilon^2\int_{t_{k-1}}^{t_k}\Phi_1^\varepsilon(s)\,\df s\notag\\
&\quad +R_{k-1}\left(\frac{\varepsilon^{2}}{n^2}\right)+R_{k-1}\left(\frac{\varepsilon^3}{n\sqrt{n}}\right).\label{R2}
\end{align}
It is satisfied from the proof of (i) that
\begin{align*}
E\left[\left(\Delta_sX^\varepsilon\right)^{i_{1}}\Big|\mathcal{F}_{t_{k-1}}\right] &=(s-t_{k-1})b_{t_{k-1},H}^{i_{1}}+\int_{t_{k-1}}^s\Phi_1^\varepsilon(u)\,\df u\\
&\quad +R_{k-1}\left(\left(s-t_{k-1}\right)^{2}\right)+R_{k-1}\left(\varepsilon\left(s-t_{k-1}\right)^{3/2}\right).
\end{align*} 
Therefore,
\begin{align}
E\left[\int_{t_{k-1}}^{t_{k}}\left(\Delta_sX^\varepsilon\right)^{i_{1}}\,\df s\bigg|\mathcal{F}_{t_{k-1}}\right]b_{t_{k-1},H}^{i_{2}}&= \frac{1}{2n^{2}}b_{t_{k-1},H}^{i_{1}}b_{t_{k-1},H}^{i_{2}}+b_{t_{k-1},H}^{i_{2}}\int_{t_{k-1}}^{t_k}\int_{t_{k-1}}^{s}\Phi_1^\varepsilon(u)\,\df u\,\df s\notag\\
&\quad +R_{k-1}\left(n^{-3}\right)+R_{k-1}\left(\varepsilon n^{-5/2}\right).\label{R3}
\end{align}
It follows from \eqref{R1}-\eqref{R3} that
\begin{align*}
E\left[\prod_{j=1}^2\left(\Delta_kX^\varepsilon\right)^{i_j}\Bigg|\mathcal{F}_{t_{k-1}}\right]&= \frac{\varepsilon^{2}}{n}[\sigma\sigma^\top]_{t_{k-1},H}^{i_{1}i_{2}}+ \frac{1}{n^{2}}b_{t_{k-1},H}^{i_{1}}b_{t_{k-1},H}^{i_{2}}+\int_{t_{k-1}}^{t_k}\left\{\Phi_2^\varepsilon(s)+\varepsilon^2\Phi_1^\varepsilon(s)\right\}\,\df s\\
&\quad +\left\{b_{t_{k-1},H}^{i_{1}}+b_{t_{k-1},H}^{i_{2}}\right\}\int_{t_{k-1}}^{t_k}\int_{t_{k-1}}^{s}\Phi_1^\varepsilon(u)\,\df u\,\df s\\
&\quad +R_{k-1}\left(\frac{1}{n^3}\right)+R_{k-1}\left(\frac{\varepsilon}{n^2\sqrt{n}}\right)+R_{k-1}\left(\frac{\varepsilon^2}{n^2}\right)+R_{k-1}\left(\frac{\varepsilon^{3}}{n\sqrt{n}}\right).
\end{align*}
Therefore,
\begin{align*}
E\left[P_{k}^{i_{1}}P_{k}^{i_{2}}(\theta_{0})\Big|\mathcal{F}_{t_{k-1}}\right]&= E\left[\left(\Delta_kX^\varepsilon\right)^{i_{1}}\left(\Delta_kX^\varepsilon\right)^{i_{2}}\Big|\mathcal{F}_{t_{k-1}}\right]\\
&\quad -\sum_{\tilde{\sigma}\in G(2)}E\left[\left(\Delta_kX^\varepsilon\right)^{i_{\tilde{\sigma}(1)}}\Big|\mathcal{F}_{t_{k-1}}\right]\frac{1}{n}b_{t_{k-1},H_n}^{i_{\tilde{\sigma}(2)}}+\frac{1}{n^{2}}b_{t_{k-1},H_n}^{i_{1}}b_{t_{k-1},H_n}^{i_{2}}\\
&= \frac{\varepsilon^{2}}{n}[\sigma\sigma^\top]_{t_{k-1},H}^{i_{1}i_{2}}+ \frac{1}{n^{2}}b_{t_{k-1},H}^{i_{1}}b_{t_{k-1},H}^{i_{2}}+\int_{t_{k-1}}^{t_k}\left\{4\Phi_2^\varepsilon(s)+\varepsilon^2\Phi_1^\varepsilon(s)\right\}\,\df s\\
&\quad +\left(b_{t_{k-1},H}^{i_{1}}+b_{t_{k-1},H}^{i_{2}}\right)\int_{t_{k-1}}^{t_k}\int_{t_{k-1}}^{s}\Phi_1^\varepsilon(u)\,\df u\,\df s\\
&\quad +R_{k-1}\left(\frac{1}{n^3}\right)+R_{k-1}\left(\frac{\varepsilon}{n^2\sqrt{n}}\right)+R_{k-1}\left(\frac{\varepsilon^2}{n^2}\right)+R_{k-1}\left(\frac{\varepsilon^{3}}{n\sqrt{n}}\right)\\
&\quad -\sum_{\tilde{\sigma}\in G(2)}\frac{1}{n}b_{t_{k-1},H_n}^{i_{\tilde{\sigma}(2)}}\bigg\{\frac{1}{n}b_{t_{k-1},H}^{i_{\tilde{\sigma}(1)}}+\int_{t_{k-1}}^{t_k}\Phi_1^\varepsilon(s)\,\df s\\
&\qquad +R_{k-1}\left(\frac{1}{n^2}\right)+R_{k-1}\left(\frac{\varepsilon}{n\sqrt{n}}\right)\bigg\}\\
&\quad +\frac{1}{n^{2}}b_{t_{k-1},H_n}^{i_{1}}b_{t_{k-1},H_n}^{i_{2}}\\
&= \frac{\varepsilon^{2}}{n}[\sigma\sigma^\top]_{t_{k-1},H}^{i_{1}i_{2}}+\frac{1}{n^2}\left(b_{t_{k-1},H}^{i_1}-b_{t_{k-1},H_n}^{i_1}\right)\left(b_{t_{k-1},H}^{i_2}-b_{t_{k-1},H_n}^{i_2}\right)\\
&\quad +\int_{t_{k-1}}^{t_k}\left\{\Phi_2^\varepsilon(s)+\varepsilon^2\Phi_1^\varepsilon(s)\right\}\,\df s\\
&\quad +R_{k-1}(1)\int_{t_{k-1}}^{t_k}\int_{t_{k-1}}^{s}\Phi_1^\varepsilon(u)\,\df u\,\df s+R_{k-1}\left(\frac{1}{n}\right)\int_{t_{k-1}}^{t_k}\Phi_1^\varepsilon(s)\,\df s\\
&\quad +R_{k-1}\left(\frac{1}{n^3}\right)+R_{k-1}\left(\frac{\varepsilon}{n^2\sqrt{n}}\right)+R_{k-1}\left(\frac{\varepsilon^2}{n^2}\right)+R_{k-1}\left(\frac{\varepsilon^{3}}{n\sqrt{n}}\right).
\end{align*}
(iii) From the same argument as the proof of (ii), it {\blue hold} that
\begin{align*}
E\left[\prod_{j=1}^3\left(\Delta_kX^\varepsilon\right)^{i_j}\Bigg|\mathcal{F}_{t_{k-1}}\right]&= \frac{\varepsilon^2}{2n^2}\sum_{\tilde{\sigma}\in G(3)}[\sigma\sigma^\top]_{t_{k-1},H}^{i_{\tilde{\sigma}(1)}i_{\tilde{\sigma}(2)}}b_{t_{k-1},H}^{i_{\tilde{\sigma}(3)}}+\frac{1}{n^3}b_{t_{k-1},H}^{i_{1}}b_{t_{k-1},H}^{i_{2}}b_{t_{k-1},H}^{i_{3}}\\
&\quad +\int_{t_{k-1}}^{t_k}\{\Phi_3^\varepsilon(s)+\varepsilon^2\Phi_2^\varepsilon(s)\}\,\df s\\
&\quad +R_{k-1}(1)\int_{t_{k-1}}^{t_k}\int_{t_{k-1}}^{s}\{\Phi_2^\varepsilon(u)+\varepsilon^2\Phi_1^\varepsilon(u)\}\,\df u\,\df s\\
&\quad +R_{k-1}(1)\int_{t_{k-1}}^{t_k}\int_{t_{k-1}}^{s}\int_{t_{k-1}}^{u}\Phi_1^\varepsilon(v)\,\df v\,\df u\,\df s\\
&\quad +R_{k-1}\left(\frac{1}{n^4}\right)+R_{k-1}\left(\frac{\varepsilon}{n^3\sqrt{n}}\right)+R_{k-1}\left(\frac{\varepsilon^2}{n^3}\right)\\
&\quad +R_{k-1}\left(\frac{\varepsilon^3}{n^2\sqrt{n}}\right)+R_{k-1}\left(\frac{\varepsilon^4}{n^2}\right),
\end{align*}
and
\begin{align*}
E\left[P_{k}^{i_{1}}P_{k}^{i_{2}}P_{k}^{i_{3}}(\theta_{0})\Big|\mathcal{F}_{t_{k-1}}\right]&= E\left[\prod_{j=1}^3\left(\Delta_kX^\varepsilon\right)^{i_j}\Bigg|\mathcal{F}_{t_{k-1}}\right]-E[(\Delta_kX^\varepsilon)^{i_{1}}(\Delta_kX^\varepsilon)^{i_{3}}|\mathcal{F}_{t_{k-1}}]\frac{1}{n}b_{t_{k-1},H_n}^{i_{2}}\\
&\quad -E[(\Delta_kX^\varepsilon)^{i_{2}}(\Delta_kX^\varepsilon)^{i_{3}}|\mathcal{F}_{t_{k-1}}]\frac{1}{n}b_{t_{k-1},H_n}^{i_{1}}\\
&\quad +E[(\Delta_kX^\varepsilon)^{i_{3}}|\mathcal{F}_{t_{k-1}}]\frac{1}{n^2}b_{t_{k-1},H_n}^{i_{1}}b_{t_{k-1},H_n}^{i_2}\\
&\quad -E[P_{k}^{i_1}P_{k}^{i_2}(\theta_{0})|\mathcal{F}_{t_{k-1}}]\frac{1}{n}b_{t_{k-1},H_n}^{i_3}\\
&= \frac{1}{n^3}\prod_{j=1}^3\left(b_{t_{k-1},H}^{i_j}-b_{t_{k-1},H_n}^{i_j}\right)\\
&\quad+\frac{\varepsilon^2}{2n^2}\sum_{\tilde{\sigma}\in G(3)}[\sigma\sigma^\top]^{i_{\tilde{\sigma}(1)}i_{\tilde{\sigma}(2)}}\left(b_{t_{k-1},H}^{i_{\tilde{\sigma}(3)}}-b_{t_{k-1},H_n}^{i_{\tilde{\sigma}(3)}}\right)\\
&\quad +\int_{t_{k-1}}^{t_k}\{\Phi_3^\varepsilon(s)+\varepsilon^2\Phi_2^\varepsilon(s)\}\,\df s+R_{k-1}\left(\frac{1}{n^2}\right)\int_{t_{k-1}}^{t_k}\Phi_1^\varepsilon(s)\,\df s\\
&\quad +R_{k-1}\left(\frac{1}{n}\right)\int_{t_{k-1}}^{t_k}\{\Phi_2^\varepsilon(s)+\varepsilon^2\Phi_1^\varepsilon(s)\}\,\df s\\
&\quad +R_{k-1}(1)\int_{t_{k-1}}^{t_k}\int_{t_{k-1}}^{s}\{\Phi_2^\varepsilon(u)+\varepsilon^2\Phi_1^\varepsilon(u)\}\,\df u\,\df s\\
&\quad +R_{k-1}\left(\frac{1}{n}\right)\int_{t_{k-1}}^{t_k}\int_{t_{k-1}}^s\Phi_1^\varepsilon(u)\,\df u\,\df s\\
&\quad +R_{k-1}(1)\int_{t_{k-1}}^{t_k}\int_{t_{k-1}}^{s}\int_{t_{k-1}}^{u}\Phi_1^\varepsilon(v)\,\df v\,\df u\,\df s\\
&\quad +R_{k-1}\left(\frac{1}{n^4}\right)+R_{k-1}\left(\frac{\varepsilon}{n^3\sqrt{n}}\right)+R_{k-1}\left(\frac{\varepsilon^2}{n^3}\right)\\
&\quad +R_{k-1}\left(\frac{\varepsilon^3}{n^2\sqrt{n}}\right)+R_{k-1}\left(\frac{\varepsilon^4}{n^2}\right).
\end{align*}
(iv) It follows from the same argument as the proof of (ii) that
\begin{align*}
E\left[\prod_{j=1}^4\left(\Delta_kX^\varepsilon\right)^{i_j}\Bigg|\mathcal{F}_{t_{k-1}}\right]&= \frac{\varepsilon^4}{8n^2}\sum_{\tilde{\sigma}\in G(4)}[\sigma\sigma^\top]_{t_{k-1},H}^{i_{\tilde{\sigma}(1)}i_{\tilde{\sigma}(2)}}[\sigma\sigma^\top]_{t_{k-1},H}^{i_{\tilde{\sigma}(3)}i_{\tilde{\sigma}(4)}}\\
&\quad +\frac{\varepsilon^2}{4n^3}\sum_{\tilde{\sigma}\in G(4)}[\sigma\sigma^{\top}]_{t_{k-1},H}^{i_{\tilde{\sigma}(1)}i_{\tilde{\sigma}(2)}}b_{t_{k-1},H}^{i_{\tilde{\sigma}(3)}}b_{t_{k-1},H}^{i_{\tilde{\sigma}(4)}}+\frac{1}{n^4}\prod_{j=1}^4b_{t_{k-1},H}^{i_j}\\
&\quad +\int_{t_{k-1}}^{t_k}\left\{\Phi_4^\varepsilon(s)+\varepsilon^2\Phi_3^\varepsilon(s)\right\}\,\df s\\
&\quad +R_{k-1}(1)\int_{t_{k-1}}^{t_k}\int_{t_{k-1}}^{s}\{\Phi_3^\varepsilon(u)+\varepsilon^2\Phi_2^\varepsilon(u)\}\,\df u\,\df s\\
&\quad +R_{k-1}(\varepsilon^2)\int_{t_{k-1}}^{t_k}\int_{t_{k-1}}^{s}\{\Phi_2^\varepsilon(u)+\varepsilon^2\Phi_1^\varepsilon(u)\}\,\df u\,\df s\\
&\quad +R_{k-1}(1)\int_{t_{k-1}}^{t_k}\int_{t_{k-1}}^{s}\int_{t_{k-1}}^{u}\left\{\Phi_2^\varepsilon(v)+\varepsilon^2\Phi_1^\varepsilon(v)\right\}\,\df v\,\df u\,\df s\\
&\quad +R_{k-1}(1)\int_{t_{k-1}}^{t_k}\int_{t_{k-1}}^s\int_{t_{k-1}}^u\int_{t_{k-1}}^v\Phi_1^\varepsilon(w)\,\df w\,\df v\,\df u\,\df s\\
&\quad +R_{k-1}\Big(\frac{1}{n^5}\Big)+R_{k-1}\Big(\frac{\varepsilon}{n^4\sqrt{n}}\Big)+R_{k-1}\Big(\frac{\varepsilon^2}{n^4}\Big)\\
&\quad +R_{k-1}\Big(\frac{\varepsilon^3}{n^3\sqrt{n}}\Big)+R_{k-1}\Big(\frac{\varepsilon^4}{n^3}\Big)+R_{k-1}\Big(\frac{\varepsilon^5}{n^2\sqrt{n}}\Big),
\end{align*}
and
\begin{align*}
E\left[P_{k}^{i_{1}}P_{k}^{i_{2}}P_{k}^{i_{3}}P_{k}^{i_{4}}(\theta_{0})\Big|\mathcal{F}_{t_{k-1}}\right]&= E\left[\prod_{j=1}^4\left(\Delta_kX^\varepsilon\right)^{i_j}\Bigg|\mathcal{F}_{t_{k-1}}\right]\\
&\quad -E[(\Delta_kX^\varepsilon)^{i_{1}}(\Delta_kX^\varepsilon)^{i_{2}}(\Delta_kX^\varepsilon)^{i_{4}}|\mathcal{F}_{t_{k-1}}]\frac{1}{n}b_{t_{k-1},H_n}^{i_{3}}\\
&\quad -E[(\Delta_kX^\varepsilon)^{i_{1}}(\Delta_kX^\varepsilon)^{i_{3}}(\Delta_kX^\varepsilon)^{i_{4}}|\mathcal{F}_{t_{k-1}}]\frac{1}{n}b_{t_{k-1},H_n}^{i_{2}}\\
&\quad -E[(\Delta_kX^\varepsilon)^{i_{2}}(\Delta_kX^\varepsilon)^{i_{3}}(\Delta_kX^\varepsilon)^{i_{4}}|\mathcal{F}_{t_{k-1}}]\frac{1}{n}b_{t_{k-1},H_n}^{i_{1}}\\
&\quad +E[(\Delta_kX^\varepsilon)^{i_{1}}(\Delta_kX^\varepsilon)^{i_{4}}|\mathcal{F}_{t_{k-1}}]\frac{1}{n^2}b_{t_{k-1},H_n}^{i_{2}}b_{t_{k-1},H_n}^{i_{3}}\\
&\quad +E[(\Delta_kX^\varepsilon)^{i_{2}}(\Delta_kX^\varepsilon)^{i_{4}}|\mathcal{F}_{t_{k-1}}]\frac{1}{n^2}b_{t_{k-1},H_n}^{i_{1}}b_{t_{k-1},H_n}^{i_{3}}\\
&\quad +E[(\Delta_kX^\varepsilon)^{i_{3}}(\Delta_kX^\varepsilon)^{i_{4}}|\mathcal{F}_{t_{k-1}}]\frac{1}{n^2}b_{t_{k-1},H_n}^{i_{1}}b_{t_{k-1},H_n}^{i_{2}}\\
&\quad -E[(\Delta_kX^\varepsilon)^{i_{4}}|\mathcal{F}_{t_{k-1}}]\frac{1}{n^3}b_{t_{k-1},H_n}^{i_{1}}b_{t_{k-1},H_n}^{i_{2}}b_{t_{k-1},H_n}^{i_3}\\
&\quad -E[P_{k}^{i_{1}}P_{k}^{i_{2}}P_{k}^{i_{3}}(\theta_{0})|\mathcal{F}_{t_{k-1}}]\frac{1}{n}b_{t_{k-1},H_n}^{i_4}\\
&= \frac{\varepsilon^4}{8n^2}\sum_{\tilde{\sigma}\in G(4)}[\sigma\sigma^\top]_{t_{k-1},H}^{i_{\tilde{\sigma}(1)}i_{\tilde{\sigma}(2)}}[\sigma\sigma^\top]_{t_{k-1},H}^{i_{\tilde{\sigma}(3)}i_{\tilde{\sigma}(4)}}\\
&\quad +\frac{1}{n^4}\prod_{j=1}^4\left(b_{t_{k-1},H}^{i_j}-b_{t_{k-1},H_n}^{i_j}\right)\\
&\quad +\frac{\varepsilon^2}{4n^3}\sum_{\tilde{\sigma}\in G(4)}[\sigma\sigma^{\top}]^{i_{\tilde{\sigma}(1)}i_{\tilde{\sigma}(2)}}\left(b_{t_{k-1},H}^{i_{\tilde{\sigma}(3)}}-b_{t_{k-1},H_n}^{i_{\tilde{\sigma}(3)}}\right)\left(b_{t_{k-1},H}^{i_{\tilde{\sigma}(4)}}-b_{t_{k-1},H_n}^{i_{\tilde{\sigma}(4)}}\right)\\
&\quad +\int_{t_{k-1}}^{t_k}\left\{\Phi_4^\varepsilon(s)+\varepsilon^2\Phi_3^\varepsilon(s)\right\}\,\df s\\
&\quad +R_{k-1}\left(\frac{1}{n}\right)\int_{t_{k-1}}^{t_k}\{\Phi_3^\varepsilon(u)+\varepsilon^2\Phi_2^\varepsilon(s)\}\,\df s\\
&\quad +R_{k-1}\left(\frac{1}{n^2}\right)\int_{t_{k-1}}^{t_k}\{\Phi_2^\varepsilon(s)+\varepsilon^2\Phi_1^\varepsilon(s)\}\,\df s\\
&\quad +R_{k-1}(1)\int_{t_{k-1}}^{s}\{\Phi_3^\varepsilon(u)+\varepsilon^2\Phi_2^\varepsilon(u)\}\,\df u\,\df s\\
&\quad +R_{k-1}\left(\frac{1}{n}\right)\int_{t_{k-1}}^{t_k}\int_{t_{k-1}}^{s}\{\Phi_2^\varepsilon(u)+\varepsilon^2\Phi_1^\varepsilon(u)\}\,\df u\,\df s\\
&\quad +R_{k-1}(\varepsilon^2)\int_{t_{k-1}}^{t_k}\int_{t_{k-1}}^{s}\{\Phi_2^\varepsilon(u)+\varepsilon^2\Phi_1^\varepsilon(u)\}\,\df u\,\df s\\
&\quad +R_{k-1}(1)\int_{t_{k-1}}^{t_k}\int_{t_{k-1}}^{s}\int_{t_{k-1}}^{u}\left\{\Phi_2^\varepsilon(v)+\varepsilon^2\Phi_1^\varepsilon(v)\right\}\,\df v\,\df u\,\df s\\
&\quad +R_{k-1}(1)\int_{t_{k-1}}^{t_k}\int_{t_{k-1}}^s\int_{t_{k-1}}^u\int_{t_{k-1}}^v\Phi_1^\varepsilon(w)\,\df w\,\df v\,\df u\,\df s\\
&\quad +R_{k-1}\left(\frac{1}{n^5}\right)+R_{k-1}\left(\frac{\varepsilon}{n^4\sqrt{n}}\right)+R_{k-1}\left(\frac{\varepsilon^2}{n^4}\right)\\
&\quad +R_{k-1}\left(\frac{\varepsilon^3}{n^3\sqrt{n}}\right)+R_{k-1}\left(\frac{\varepsilon^4}{n^3}\right)+R_{k-1}\Big(\frac{\varepsilon^5}{n^2\sqrt{n}}\Big).
\end{align*}
\end{proof}
\end{lemm}
\noindent
We shall use the notation:
\begin{enumerate}
\renewcommand{\labelenumi}{(N\theenumi)}
\setcounter{enumi}{\value{enumimemory}}
\item We denote the gradient operator of $f(x,y,\theta)$ with respect to $x$ by
\begin{equation*}
\triangledown_xf(x,y,\theta)=\left(\partial_{x_1}f(x,y,\theta),\dots\partial_{x_d}f(x,y,\theta)\right)^{\top}.
\end{equation*}
\setcounter{enumimemory}{\value{enumi}}
\end{enumerate}
In Lemma \ref{convergence depend on n}, the proof ideas follow Long et al. \cite{long}
\begin{lemm} \label{convergence depend on n}
Let $f\in C_{\uparrow}^{1,1,1}(\mathbb{R}^d\times\mathbb{R}^d\times\Theta)$ and suppose the conditions {\upshape(A\ref{assu1})-(A\ref{assu3})}. Then the following (i) and (ii) hold true: 
\begin{enumerate}
\renewcommand{\labelenumi}{\rm{(\roman{enumi})}}
\item As $\varepsilon\to 0$ and $n\rightarrow \infty$, 
\begin{equation*}
\frac{1}{n}\sum_{k=1}^{n}f\left(X_{t_{k-1}}^{\varepsilon},H_n\big(X_{t_{k-1}-\cdot}^{\varepsilon}\big),\theta\right)\xrightarrow{P}\int_{0}^{1}f\left(X_{s}^0,H\big(X_{s-\cdot}^0\big),\theta\right)\,\df s,
\end{equation*}
uniformly in $\theta\in \Theta$. 
\item As $\varepsilon\to 0$ and $n\rightarrow \infty$, 
\begin{equation*}
\sum_{k=1}^{n}f\left(X_{t_{k-1}}^{\varepsilon},H_n\big(X_{t_{k-1}-\cdot}^{\varepsilon}\big),\theta\right)P_k(\theta_0)\xrightarrow{P}0,
\end{equation*}
uniformly in $\theta\in \Theta$. 
\end{enumerate}
\end{lemm}

\begin{proof}[\bf Proof]
(i) From lemma \ref{X_t go to ODE}, Lemma \ref{H(Y) to H(x) and H_n(Y) to H(x)} and Taylor's formula, we find that
\begin{align*}
&\sup_{\theta\in\Theta}\left|\frac{1}{n}\sum_{k=1}^{n}f\left(X_{t_{k-1}}^\varepsilon,H_n\big(X_{t_{k-1}-\cdot}^\varepsilon\big),\theta\right)-\int_{0}^{1}f\left(X_{s}^0,H\big(X_{s-\cdot}^0\big),\theta\right)\,\df s\right|\\
=&\sup_{\theta\in\Theta}\left|\int_{0}^{1}\left\{f\left(Y_{s}^{n,\varepsilon},H_n\big(Y_{s-\cdot}^{n,\varepsilon}\big),\theta\right)-f\left(X_{s}^0,H\big(X_{s-\cdot}^0\big),\theta\right)\right\}\,\df s\right|\\
\leq &\sup_{\theta\in\Theta}\int_0^1\int_0^1\Big|(\triangledown_xf)^\top\left(X_s^0+u\left(Y_s^{n,\varepsilon}-X_s^0\right),H\big(X_{s-\cdot}^0\big)+u\left(H_n\big(Y_{s-\cdot}^{n,\varepsilon}\big)-H\big(X_{s-\cdot}^0\big)\right),\theta\right)\\
&\quad \cdot\left(Y_s^{n,\varepsilon}-X_s^0\right)+(\triangledown_yf)^\top\left(X_s^0+u\left(Y_s^{n,\varepsilon}-X_s^0\right),H\big(X_{s-\cdot}^0\big)+u\left(H_n\big(Y_{s-\cdot}^{n,\varepsilon}\big)-H\big(X_{s-\cdot}^0\big)\right),\theta\right)\\
&\quad \cdot\left(H_n\big(Y_{s-\cdot}^{n,\varepsilon}\big)-H\big(X_{s-\cdot}^0\big)\right)\Big|\,\df u\,\df s\\
\leq &\int_0^1C\left\{1+\left|X_s^0\right|+\left|Y_s^{n,\varepsilon}\right|+\mu\left(\left[0,\delta\right]\right)\left(\left\|X_{s-\cdot}^0\right\|_\infty+\left\|Y_{s-\cdot}^{n,\varepsilon}\right\|_\infty\right)\right\}^\lambda\Big(\left|Y_s^{n,\varepsilon}-X_s^0\right|\\
&\quad +\left|H_n\big(Y_{s-\cdot}^{n,\varepsilon}\big)-H\big(Y_{s-\cdot}^{n,\varepsilon}\big)\right|+\mu\left(\left[0,\delta\right]\right)\left\|Y_{s-\cdot}^{n,\varepsilon}-X_{s-\cdot}^0\right\|_\infty\Big)\,\df s\\
\leq &~C\Big(1+\big(1+\mu([0,\delta])\big)\big(\sup_{-\delta\leq s\leq1}|X_{s}^0|+\sup_{-\delta\leq s\leq 1}|Y_{s}^{n,\varepsilon}|\big)\Big)^\lambda\Big\{\sup_{0\leq s\leq1}|Y_s^{n,\varepsilon}-X_s^0|\\
&\quad +\sup_{0\leq s\leq1}|H_n(Y_{s-\cdot}^{n,\varepsilon})-H(Y_{s-\cdot}^{n,\varepsilon})|+\mu([0,\delta])\sup_{0\leq s\leq1}\|Y_{s-\cdot}^{n,\varepsilon}-X_{s-\cdot}^0\|_\infty\Big\}\\
\xrightarrow{P} &0,
\end{align*}
as $\varepsilon\to 0$ and $n\rightarrow \infty.$\\
(ii) It is easy to see that
\begin{align*}
&\sum_{k=1}^{n}f(X_{t_{k-1}}^\varepsilon,H_n(X_{t_{k-1}-\cdot}^\varepsilon),\theta)P_k^i(\theta_0)\\
=&\sum_{k=1}^{n}f(X_{t_{k-1}}^\varepsilon,H_n(X_{t_{k-1}-\cdot}^\varepsilon),\theta)\int_{t_{k-1}}^{t_k}\Big(b^i\big(X_{s}^\varepsilon,H(X_{s-\cdot}^\varepsilon),\theta_0\big)-b^i\big(X_{t_{k-1}}^\varepsilon,H_n(X_{t_{k-1}-\cdot}^\varepsilon),\theta_0\big)\Big)\,\df s\\
&+\varepsilon\sum_{k=1}^{n}f(X_{t_{k-1}}^\varepsilon,H_n(X_{t_{k-1}-\cdot}^\varepsilon),\theta) \int_{t_{k-1}}^{t_k}\sum_{j=1}^{r}\sigma^{ij}\big(X_{s}^\varepsilon,H(X_{s-\cdot}^\varepsilon),\beta_0\big)\,\df W_s^j\\
=&\int_{0}^{1}f(Y_{s}^{n,\varepsilon},H_n(Y_{s-\cdot}^{n,\varepsilon}),\theta)\Big(b\big(X_{s}^\varepsilon,H(X_{s-\cdot}^\varepsilon),\theta_0\big)-b\big(Y_{s}^{n,\varepsilon},H_n(Y_{s-\cdot}^{n,\varepsilon}),\theta_0\big)\Big)\,\df s\\
&+\varepsilon\int_{0}^{1}\sum_{j=1}^{r}f(Y_{s}^{n,\varepsilon},H(Y_{s-\cdot}^{n,\varepsilon}),\theta) \sigma^{ij}\big(X_{s}^\varepsilon,H(X_{s-\cdot}^\varepsilon),\beta_0\big)\,\df W_s^j.
\end{align*}
From the Lipschitz condition on $b$ in (A\ref{assu1}) it holds that
\begin{align*}
&\sup_{\theta\in\Theta}\bigg|\int_{0}^{1}f(Y_{s}^{n,\varepsilon},H_n(Y_{s-\cdot}^{n,\varepsilon}),\theta)\Big(b\big(X_{s}^\varepsilon,H(X_{s-\cdot}^\varepsilon),\theta_0\big)-b\big(Y_{s}^{n,\varepsilon},H_n(Y_{s-\cdot}^{n,\varepsilon}),\theta_0\big)\Big)\,\df s\bigg|\\
\leq &\int_{0}^{1}\sup_{\theta\in\Theta}\Big|f(Y_{s}^{n,\varepsilon},H_n(Y_{s-\cdot}^{n,\varepsilon}),\theta)\Big|\cdot K\Big(\big|X_{s}^\varepsilon-Y_s^{n,\varepsilon}\big|\\
&\quad +\big|H_n(Y_{s-\cdot}^{n,\varepsilon})-H(Y_{s-\cdot}^{n,\varepsilon})\big|+\mu([0,\delta])\big\|X_{s-\cdot}^\varepsilon-Y_{s-\cdot}^{n,\varepsilon}\big\|_\infty \Big)\,\df s\\
\leq &KC\int_{0}^{1}\big(1+|Y_{s}^{n,\varepsilon}|+\mu([0,\delta])\|Y_{s-\cdot}^{n,\varepsilon}\|_\infty\big)^\lambda\Big(\big|X_{s}^\varepsilon-X_s^0\big|+\big|Y_s^{n,\varepsilon}-X_s^0\big|\\
&\quad +\big|H_n(Y_{s-\cdot}^{n,\varepsilon})-H(Y_{s-\cdot}^{n,\varepsilon})\big|+\mu([0,\delta])\big\|X_{s-\cdot}^\varepsilon-X_{s-\cdot}^0\big\|+\mu([0,\delta])\big\|Y_{s-\cdot}^{n,\varepsilon}-X_{s-\cdot}^0\big\|\Big)\,\df s\\
\leq &KC\Big(1+\left(1+\mu([0,\delta])\right)\sup_{-\delta \leq s \leq 1}\big|X_{s}^\varepsilon\big|\Big)^\lambda\Big\{\sup_{0 \leq s \leq 1}\big|H_n(Y_{s-\cdot}^{n,\varepsilon})-H(Y_{s-\cdot}^{n,\varepsilon})\big|\\
&\quad +\left(1+\mu([0,\delta])\right)(\sup_{-\delta \leq s \leq 1}\big|X_{s}^\varepsilon-X_s^0\big|+\sup_{-\delta \leq s \leq 1}\big|Y_s^{n,\varepsilon}-X_s^0\big|)\Big\},
\end{align*}
which converges to zero as $\varepsilon\to 0$ and $n\to\infty$ by Lemma \ref{X_t go to ODE} and Lemma \ref{H(Y) to H(x) and H_n(Y) to H(x)}. Let $\tau_m^{n,\varepsilon}=\inf\{t\geq0;|X_t^\varepsilon|\geq m$ or $|Y_t^{n,\varepsilon}|\geq m\}$. We find that $\tau_m^{n,\varepsilon}\to\infty$~ a.s.  as $m\to \infty$ from Lemma \ref{having solution}. Next, we have
\begin{align}
&P\Bigg(\varepsilon\sup_{\theta\in\Theta}\bigg|\int_{0}^{1}\sum_{j=1}^{r}f(Y_{s}^{n,\varepsilon},H_n(Y_{s-\cdot}^{n,\varepsilon}),\theta) \sigma^{ij}\big(X_{s}^\varepsilon,H(X_{s-\cdot}^\varepsilon),\beta_0\big)\,\df W_s^j\bigg|\Bigg)\notag\\
\leq &P(\tau_{m}^{n,\varepsilon}<1)+P\Bigg(\varepsilon\sup_{\theta\in\Theta}\bigg|\int_{0}^{1}\sum_{j=1}^{r}f(Y_{s}^{n,\varepsilon},H_n(Y_{s-\cdot}^{n,\varepsilon}),\theta) \sigma^{ij}\big(X_{s}^\varepsilon,H(X_{s-\cdot}^\varepsilon),\beta_0\big)\1_{\{s\leq\tau_m^{n,\varepsilon}\}}\,\df W_s^j\bigg|\Bigg).\label{probability  convergence}
\end{align}
Let
\begin{equation*}
u_{n,\varepsilon}^i(\theta)=\varepsilon\int_{0}^{1}\sum_{j=1}^{r}f(Y_{s}^{n,\varepsilon},H_n(Y_{s-\cdot}^{n,\varepsilon}),\theta) \sigma^{ij}\big(X_{s}^\varepsilon,H(X_{s-\cdot}^\varepsilon),\beta_0\big)\1_{\{s\leq\tau_m^{n,\varepsilon}\}}\,\df W_s^j.
\end{equation*}
We want to prove that $u_{n,\varepsilon}^i(\theta)\to0$ in $P$ as $\varepsilon\to 0$ and $n\to\infty$, uniformly in $\theta\in\Theta$. Therefore, it is sufficient to check the pointwise convergence and the tightness of the sequence $\{u_{n,\varepsilon}^i(\cdot)\}$. For the pointwise convergence, by the Chebyshev's inequality, the linear growth condition on $\sigma$ in (A\ref{assu1}) and it\^{o}'s isometry,
\begin{align}
P\left(|u_{n,\varepsilon}^i(\theta)|>\eta\right)&\leq \varepsilon^2\eta^{-2}E\Bigg[\bigg|\int_{0}^{1}\sum_{j=1}^{r}f(Y_{s}^{n,\varepsilon},H_n(Y_{s-\cdot}^{n,\varepsilon}),\theta) \sigma^{ij}\big(X_{s}^\varepsilon,H(X_{s-\cdot}^\varepsilon),\beta_0\big)\1_{\{s\leq\tau_m^{n,\varepsilon}\}}\,\df W_s^j\bigg|^2\Bigg]\notag\\
&\leq \varepsilon^2\eta^{-2}\sum_{j=1}^{r}\int_{0}^{1}E\bigg[\Big|f(Y_{s}^{n,\varepsilon},H_n(Y_{s-\cdot}^{n,\varepsilon}),\theta) \sigma^{ij}\big(X_{s}^\varepsilon,H(X_{s-\cdot}^\varepsilon),\beta_0\big)\Big|^2\1_{\{s\leq\tau_m^{n,\varepsilon}\}}\bigg]\,\df s\notag\\
&\leq \varepsilon^2\eta^{-2}C^2K^2r\Big(1+\big(1+\mu([0,\delta])\big)m\Big)^{2(\lambda+1)}.\label{pointwise convergence}
\end{align}
which converges to zero as $\varepsilon\to 0$ and $n\to\infty$ with fixed $m$. For the tightness, by using Theorem 20 in Appendix\Rnum{1} of Ibragimov and Has'minskii \cite{Ibra}, it is adequate to prove the following two inequalities:
\begin{align}
&E\left[\left|u_{n,\varepsilon}^{i}(\theta)\right|^{2l}\right]\leq C,\label{check1}\\
&E\left[\left|u_{n,\varepsilon}^{i}(\theta_2)-u_{n,\varepsilon}^{i}(\theta_1)\right|^{2l}\right]\leq C\left|\theta_2-\theta_1\right|^{2l},\label{check2}
\end{align}
for $\theta,\theta_1,\theta_2\in\Theta$, where $2l>p+q$. The proof of \eqref{check1} is analogous to moment estimates in \eqref{pointwise convergence} by replacing it\^{o}'s isometry with the Burkholder-Davis-Gundy inequality, so we omit the detail here. For \eqref{check2}, by using Taylor's formula and the Burkholder-Davis-Gundy inequality, we get
\begin{align*}
E\left[\left|u_{n,\varepsilon}^{i}(\theta_2)-u_{n,\varepsilon}^{i}(\theta_1)\right|^{2l}\right]&\leq \varepsilon^{2l}C_lE\Bigg[\bigg(\int_{0}^{1}\sum_{j=1}^{r}\Big|\sigma^{ij}\big(X_{s}^\varepsilon,H(X_{s-\cdot}^\varepsilon),\beta_0\big)\Big|^2 \1_{\{s\leq\tau_m^{n,\varepsilon}\}}\times\\
&\quad \Big|\big(f(Y_{s}^{n,\varepsilon},H_n(Y_{s-\cdot}^{n,\varepsilon}),\theta_2)-f(Y_{s}^{n,\varepsilon},H_n(Y_{s-\cdot}^{n,\varepsilon}),\theta_1)\big)\Big|^2\,\df s\bigg)^l\Bigg]\\
&\leq \varepsilon^{2l}C_lE\Bigg[\bigg(\int_{0}^{1}\int_0^1\sum_{j=1}^{r}\Big|\sigma^{ij}\big(X_{s}^\varepsilon,H(X_{s-\cdot}^\varepsilon),\beta_0\big)\Big|^2 \1_{\{s\leq\tau_m^{n,\varepsilon}\}}\times\\
&\quad \Big|\theta_2-\theta_1\Big|^2\Big|\triangledown_{\theta}f\left(Y_{s}^{n,\varepsilon},H_n(Y_{s-\cdot}^{n,\varepsilon}),\theta_1+v(\theta_2-\theta_1)\right)\Big|^2\,\df vds\bigg)^l\Bigg]\\
&\leq \varepsilon^{2l}C_lC^{2l}K^{2l}r^{l}\Big(1+\big(1+\mu([0,\delta])\big)m\Big)^{(2\lambda+2) l}\big|\theta_2-\theta_1\big|^{2l}.
\end{align*}
Combining \eqref{probability convergence} and arguments above, we have that 
\begin{equation*}
\varepsilon\sup_{\theta\in\Theta}\bigg|\int_{0}^{1}\sum_{j=1}^{r}f(Y_{s}^{n,\varepsilon},H_n(Y_{s-\cdot}^{n,\varepsilon}),\theta) \sigma^{ij}\big(X_{s}^\varepsilon,H(X_{s-\cdot}^\varepsilon,\beta_0\big)\,\df W_s^j\bigg|,
\end{equation*}
converges to zero in probability as $\varepsilon\to 0$ and $n\to\infty$. Therefore, the proof is complete.
\end{proof}
In Lemma \ref{convergence of diffusion}, the proof ideas follow S\o rensen and Uchida \cite{sorensen}.
\begin{lemm} \label{convergence of diffusion}
Let $f\in C_{\uparrow}^{1,1,1}(\mathbb{R}^d\times\mathbb{R}^d\times\Theta)$ and suppose the conditions {\upshape(A\ref{assu1})-(A\ref{assu3})}. If $(\sqrt{n}\varepsilon)^{-1}\to0$ as $\varepsilon\to0$ and $n\to\infty$, then the following (i) and (ii) hold true:
\begin{enumerate}
\renewcommand{\labelenumi}{\rm{(\roman{enumi})}}
\item As $\varepsilon\to 0$ and $n\rightarrow \infty$, 
\begin{align*}
\varepsilon^{-2} \sum_{k=1}^{n} f\left(X_{t_{k-1}}^\varepsilon, H_n\big(X_{t_{k-1}-\cdot}^\varepsilon\big), \theta\right) P_{k}^{i}(\theta_{0}) P_{k}^{j}(\theta_{0}) \xrightarrow{P}& \int_{0}^{1} f\left(X_{s}^0, H\big(X_{s - \cdot}^0\big), \theta\right) \left[\sigma \sigma^{\top}\right]^{ij}\left(X_{s}^0, H\big(X_{s - \cdot}^0\big),\beta_{0}\right)\,\df s,
\end{align*}
uniformly in $\theta\in \Theta.$\\
\item As $\varepsilon\to 0$ and $n\rightarrow \infty$, 
\begin{align*}
&\varepsilon^{-2} \sum_{k=1}^{n} f\left(X_{t_{k-1}}^\varepsilon, H_n\big(X_{t_{k-1} - \cdot}^\varepsilon\big), \theta\right) P_{k}^{i}(\theta) P_{k}^{j}(\theta)\\ \xrightarrow{P}& \int_{0}^{1} f\left(X_{s}^0, H\big(X_{s - \cdot}^0\big), \theta\right) \left[\sigma \sigma^{\top}\right]^{ij}\left(X_{s}^0, H\big(X_{s - \cdot}^0\big),\beta_{0}\right)\,\df s,
\end{align*}
uniformly in $\theta\in \Theta.$\\
\end{enumerate}
\end{lemm}
\begin{proof}[\bf Proof]
(i) It holds from Lemma \ref{conditional expectation}, Lemma \ref{convergence depend on n}(i), and the H\"{o}lder's inequality that
\begin{align*}
\sum_{k=1}^{n}E&\left[\varepsilon^{-2} f\left(X_{t_{k-1}}, H_n\big(X_{t_{k-1}-\cdot}\big), \theta\right) P_{k}^{i}(\theta_{0}) P_{k}^{j}(\theta_{0})\Big|\mathcal{F}_{t_{k-1}}\right]\\
\xrightarrow{P}& \int_{0}^{1} f\left(X_{s}^0, H\big(X_{s - \cdot}^0\big), \theta\right) \left[\sigma \sigma^{\top}\right]^{ij}\left(X_{s}^0, H\big(X_{s - \cdot}^0\big), \beta_{0}\right)\,\df s,
\end{align*}
\begin{equation*}
\sum_{k=1}^{n}E\left[\varepsilon^{-4} f\left(X_{t_{k-1}}, H_n\big(X_{t_{k-1}-\cdot}\big), \theta\right)^2 \left(P_{k}^{i}(\theta_{0}) P_{k}^{j}(\theta_{0})\right)^2\bigg|\mathcal{F}_{t_{k-1}}\right]\xrightarrow{P} 0,
\end{equation*}
as $n \to \infty$ and $\varepsilon\to0$. Therefore, it follows from Lemma 9 in Genon-Catalot and Jacod \cite{genon-catalot1993} 
\begin{align*}
&\varepsilon^{-2} \sum_{k=1}^{n} f\left(X_{t_{k-1}}, H_n(X_{t_{k-1}-\cdot}), \theta\right) P_{k}^{i}(\theta_{0}) P_{k}^{j}(\theta_{0})\\ \xrightarrow{P}& \int_{0}^{1} f\left(X_{s}^0, H(X_{s - \cdot}^0), \theta\right) [\sigma \sigma^{\top}]^{ij}\left(X_{s}^0, H(X_{s - \cdot}^0),\beta_{0}\right)\,\df s, 
\end{align*}
as $n \to \infty$ and $\varepsilon\to0$. 
For the tightness of the sequence $\{\varepsilon^{-2}\sum_{k=1}^{n}f\left(X_{t_{k-1}},H_n(X_{t_{k-1}}),\cdot\right)P_k^iP_k^j(\theta_0)\}$, according to Lemma $\ref{conditional expectation}$,
\begin{align*}
&\sup_{n,\varepsilon}E\Bigg[\sup_{\theta}\bigg|\varepsilon^{-2} \sum_{k=1}^n\frac{\partial}{\partial \theta}f\left(X_{t_{k-1}},H_n(X_{t_{k-1}}),\theta\right) P_{k}^{i}(\theta_{0}) P_{k}^{j}(\theta_{0})\bigg|\Bigg]\\
\leq&\sup_{n,\varepsilon}E\Bigg[\frac{\varepsilon^{-2}}{2} \sum_{k=1}^n\sup_{\theta}\bigg|\frac{\partial}{\partial \theta}f\left(X_{t_{k-1}},H_n(X_{t_{k-1}}),\theta\right)\bigg| E\Big[\big(P_{k}^{i}(\theta_{0})\big)^2+ \big(P_{k}^{j}(\theta_{0})\big)^2\Big|\mathcal{F}_{t_{k-1}}\Big]\Bigg]\\
\leq&\frac{1}{2}\sup_{n,\varepsilon}E\Bigg[\sum_{k=1}^n\sup_{\theta}\bigg|\frac{\partial}{\partial \theta}f\left(X_{t_{k-1}},H_n(X_{t_{k-1}}),\theta\right)\bigg|\times\\
&\quad \bigg\{\frac{1}{n}([\sigma\sigma^{\top}]_{t_{k-1},H}^{ii}+[\sigma\sigma^{\top}]_{t_{k-1},H}^{jj})+\frac{1}{n^2\varepsilon^2}\left(b_{t_{k-1},H}^{i_1}-b_{t_{k-1},H_n}^{i_1}\right)\left(b_{t_{k-1},H}^{i_2}-b_{t_{k-1},H_n}^{i_2}\right)\\
&\qquad +\int_{t_{k-1}}^{t_k}\left\{4\varepsilon^{-2}\Phi_2^\varepsilon(s)+\Phi_1^\varepsilon(s)\right\}~ds\\
&\qquad +R_{k-1}(\varepsilon^{-2})\int_{t_{k-1}}^{t_k}\int_{t_{k-1}}^{s}\Phi_1^\varepsilon(u)\,\df u\,\df s+R_{k-1}\left(\frac{1}{n\varepsilon^2}\right)\int_{t_{k-1}}^{t_k}\Phi_1^\varepsilon(s)\,\df s\\
&\qquad+R_{k-1}\left(\frac{1}{n^3\varepsilon^2}\right)+R_{k-1}\left(\frac{\varepsilon^{-1}}{n^2\sqrt{n}}\right)+R_{k-1}\left(\frac{1}{n^2}\right)+R_{k-1}\left(\frac{\varepsilon}{n\sqrt{n}}\right)\bigg\}\Bigg] <\infty.
\end{align*}
(ii) Noticing that
\begin{align*}
P_k^iP_k^j(\theta)&=P_k^iP_k^j(\theta_0)+\frac{1}{n}P_k^i(\theta_0)B^j_{k-1}(\theta_0,\theta)+\frac{1}{n}P_k^j(\theta_0)B^i_{k-1}(\theta_0,\theta)+\frac{1}{n^2}B^i_{k-1}B^j_{k-1}(\theta_0,\theta),
\end{align*}
where $B^i_{k-1}(\theta_0,\theta)=b^i\left(X_{t_{k-1}},H_n(X_{t_{k-1}-\cdot}),\theta_0\right)-b^i\left(X_{t_{k-1}},H_n(X_{t_{k-1}-\cdot}),\theta\right)$. It follows from Lemmas \ref{convergence depend on n} and \ref{convergence of diffusion}(i), under $P$, as $n\to\infty$ and $\varepsilon\to0$,
\begin{align*}
\varepsilon^{-2} \sum_{k=1}^{n} f\left(X_{t_{k-1}}, H_n(X_{t_{k-1}-\cdot}), \theta\right) P_{k}^{i}P_{k}^{j}(\theta)&= \varepsilon^{-2} \sum_{k=1}^{n} f\left(X_{t_{k-1}}, H_n(X_{t_{k-1}-\cdot}), \theta\right)P_k^iP_k^j(\theta_0)\\
&\quad +\frac{\varepsilon^{-2}}{n^2}\sum_{k=1}^{n} f\left(X_{t_{k-1}}, H_n(X_{t_{k-1}-\cdot}), \theta\right)B^i_{k-1}B^j_{k-1}(\theta_0,\theta)\\
&\quad +\frac{1}{n\varepsilon^2}\sum_{k=1}^{n} \bigg[f\left(X_{t_{k-1}}, H_n(X_{t_{k-1}-\cdot}), \theta\right)\times\\
&\qquad \Big\{P_k^i(\theta_0)B^j_{k-1}(\theta_0,\theta)+P_k^j(\theta_0)B^i_{t_{k-1}}(\theta_0,\theta)\Big\}\bigg]\\
&\to \int_{0}^{1} f\left(X_{s}^0, H(X_{s - \cdot}^0), \theta\right) [\sigma \sigma^{\top}]^{ij}\left(X_{s}^0, H(X_{s - \cdot}^0),\beta_{0}\right)\,\df s,
\end{align*}
uniformly in $\theta\in\Theta$.
\end{proof}
\noindent
We are ready to prove Theorem \ref{consistency}. In Theorem \ref{consistency}, the proof ideas mainly follow S\o rensen and Uchida \cite{sorensen}.
\begin{proof}[\bf Proof of Theorem \ref{consistency}]
Following the proof of Theorem 1 in S\o rensen and Uchida \cite{sorensen}, the consistency follows from the two properties:
\begin{align}
&\varepsilon^{2} \{U_{n,\varepsilon}(\alpha, \beta)-U_{n,\varepsilon}(\alpha_0, \beta)\} \xrightarrow{P} \int_{0}^{1} B_s^{\top}(\theta_0, \theta) [\sigma \sigma^{\top}]^{-1}(X_{s}^0, H(X_{s - \cdot}^0),\beta_{0})B_s(\theta_0, \theta)\,\df s,\label{consistency of alpha}\\
&\frac{1}{n}U_{n,\varepsilon}(\alpha,\beta)\xrightarrow{P}\int_0^1\log\det[\sigma\sigma^{\top}]\Bigl(X_{s}^0,H(X_{s-\cdot}^0),\beta\Bigr)\,\df s\notag\\
&\qquad\qquad\qquad\qquad +\int_0^1\mathrm{tr}\bigg[[\sigma\sigma^{\top}]\Bigl(X_{s},H(X_{s-\cdot}),\beta\Bigr)[\sigma\sigma^{\top}]^{-1}\Bigl(X_{s},H(X_{s-\cdot}),\beta_{0}\Bigr)\bigg]\,\df s,\label{consistency of beta}
\end{align}
where $B_s(\theta_0, \theta)=b\left(X_s^0,H(X_{s-\cdot}^0),\theta_0\right)-b\left(X_s^0,H(X_{s-\cdot}^0),\theta\right)$, as $\varepsilon\rightarrow 0$ and $n\rightarrow \infty,$ uniformly in $\theta\in \Theta.$ First, we show \eqref{consistency of alpha}. It is clear that
\begin{align*}
\varepsilon^{2} \{U_{n,\varepsilon}(\alpha, \beta)-U_{n,\varepsilon}(\alpha_0, \beta)\}&= n\sum_{k=1}^{n}\big(P_k(\theta)-P_k(\theta_0)\big)^{\top}\Xi_{k-1}^{-1}(\beta)\big(P_k(\theta)+P_k(\theta_0)\big)\\
&= \sum_{k=1}^{n}\bigg[\Big(b\left(X_{t_{k-1}},H_n(X_{t_{k-1}-\cdot}),\theta_0\right)-b\left(X_{t_{k-1}},H_n(X_{t_{k-1}-\cdot}),\theta\right)\Big)^{\top}\\
&\quad \Xi_{k-1}^{-1}(\beta)\bigg\{2\Big(\Delta_k X-\frac{1}{n}b\left(X_{t_{k-1}},H_n(X_{t_{k-1}-\cdot}),\theta_0\right)\Big)\\
&\qquad +\frac{1}{n}\Big(b\left(X_{t_{k-1}},H_n(X_{t_{k-1}-\cdot}),\theta_0\right)-b\left(X_{t_{k-1}},H_n(X_{t_{k-1}-\cdot}),\theta\right)\Big)\bigg\}\bigg].
\end{align*}
From Lemma \ref{convergence depend on n},
\begin{equation*}
\varepsilon^{2} \{U_{n,\varepsilon}(\alpha, \beta)-U_{n,\varepsilon}(\alpha_0, \beta)\} \xrightarrow{P} \int_{0}^{1} B_s^{\top}(\theta_0, \theta) [\sigma \sigma^{\top}]^{-1}(X_{s}^0, H(X_{s - \cdot}^0),\beta_{0})B_s(\theta_0, \theta)\,\df s,
\end{equation*}
as $\varepsilon\rightarrow 0$ and $n\rightarrow \infty,$ uniformly in $\theta\in \Theta.$ About \eqref{consistency of beta}, from Lemma \ref{convergence depend on n}(i) and Lemma \ref{convergence of diffusion}(ii), as $\varepsilon\rightarrow 0$ and $n\rightarrow \infty,$ it holds that
\begin{align*}
\frac{1}{n}U_{n,\varepsilon}(\alpha,\beta)\xrightarrow{P}&\int_0^1\log\det[\sigma\sigma^{\top}]\Bigl(X_{s}^0,H(X_{s-\cdot}^0),\beta\Bigr)\,\df s\\
&\quad +\int_0^1\mathrm{tr}\bigg[[\sigma\sigma^{\top}]\Bigl(X_{s},H(X_{s-\cdot}),\beta\Bigr)[\sigma\sigma^{\top}]^{-1}\Bigl(X_{s},H(X_{s-\cdot}),\beta_{0}\Bigr)\bigg]\,\df s.
\end{align*}
\end{proof}

\noindent
Finally, we prove the asymptotic normality of $\wh{\theta}_{n,\varepsilon}$. For the proof, we shall use the notations:
\begin{enumerate}
\renewcommand{\labelenumi}{(N\theenumi)}
\setcounter{enumi}{\value{enumimemory}}
\item We denote
\begin{equation*}
\Lambda_{n,\varepsilon}(\theta_0):=
\left(
\begin{array}{cc}
-\varepsilon\left(\frac{\partial}{\partial\alpha_i}U_{n,\varepsilon}(\theta)\Big|_{\theta=\theta_0}\right)_{1\leq i\leq p} \\
-\frac{1}{\sqrt{n}}\left(\frac{\partial}{\partial\beta_i}U_{n,\varepsilon}(\theta)\Big|_{\theta=\theta_0}\right)_{1\leq i\leq q} \\
\end{array}
\right).
\end{equation*}
\item We denote
\begin{equation*}
C_{n,\varepsilon}(\theta_0):=
\left(
\begin{array}{cc}
\varepsilon^2\left(\frac{\partial^2}{\partial\alpha_i\alpha_j}U_{n,\varepsilon}(\theta)\Big|_{\theta=\theta_0}\right)_{1\leq i,j\leq p} &\frac{\varepsilon}{\sqrt{n}}\left(\frac{\partial^2}{\partial\alpha_i\beta_j}U_{n,\varepsilon}(\theta)\Big|_{\theta=\theta_0}\right)_{1\leq i\leq p,1\leq j\leq q}\\
\frac{\varepsilon}{\sqrt{n}}\left(\frac{\partial^2}{\partial\beta_i\alpha_j}U_{n,\varepsilon}(\theta)\Big|_{\theta=\theta_0}\right)_{1\leq i\leq p,1\leq j\leq q}&\frac{1}{n}\left(\frac{\partial^2}{\partial\beta_i\beta_j}U_{n,\varepsilon}(\theta)\Big|_{\theta=\theta_0}\right)_{1\leq i,j\leq q} \\
\end{array}
\right).
\end{equation*}
\end{enumerate}
In Theorem \ref{asymptotic normal}, the proof ideas mainly follow S\o rensen and Uchida \cite{sorensen}.
\begin{proof}[\bf Proof of Theorem \ref{asymptotic normal}]
By Theorem 1 in S\o rensen and Uchida\cite{sorensen}, the asymptotic normality follows from the three properties:
\begin{gather}
C_{n,\varepsilon}(\theta_0)\xrightarrow{P}2I(\theta_0),\label{check1 of a.n.}\\
\sup_{|\theta|\leq \eta_{n,\varepsilon}}|C_{n,\varepsilon}(\theta_0+\theta)-C_{n,\varepsilon}(\theta_0)|\xrightarrow{P}0, \label{check2 of a.n.}
\end{gather}
where $\eta_{n,\varepsilon}\to0$, as $\varepsilon\rightarrow 0$ and $n\rightarrow \infty$, and
\begin{equation}
\Lambda_{n,\varepsilon}(\theta_0)\xrightarrow{d} N(0,4I(\theta_{0})), \label{check3 of a.n.}
\end{equation}
as $\varepsilon\rightarrow 0$ and $n\rightarrow \infty$. First, we show \eqref{check1 of a.n.}. Note that
\begin{align}
\varepsilon^2\frac{\partial^2}{\partial\alpha_i\partial\alpha_j}U_{n,\varepsilon}(\theta)&= -2\sum_{k=1}^{n}\Bigg[\bigg(\frac{\partial^2}{\partial\alpha_i\partial\alpha_j}b_{t_{k-1},H_n}(\theta)\bigg)^{\top}\Xi_{k-1}^{-1}(\beta)\notag\\
&\qquad \Big\{P_k(\theta_0)+\frac{1}{n}\Big(b_{t_{k-1},H_n}(\theta_0)-b_{t_{k-1},H_n}(\theta)\Big)\Big\}\Bigg]\notag\\
&\quad +\frac{2}{n}\sum_{k=1}^{n}\bigg(\frac{\partial}{\partial\alpha_i}b_{t_{k-1},H_n}(\theta)\bigg)^{\top}\Xi_{k-1}^{-1}(\beta)\bigg(\frac{\partial}{\partial\alpha_j}b_{t_{k-1},H_n}(\theta)\bigg),\label{alphaalpha}
\end{align}
\begin{align}
&\frac{\varepsilon}{\sqrt{n}}\frac{\partial^2}{\partial\alpha_i\partial\beta_j}U_{n,\varepsilon}(\theta)\notag\\
= &\frac{-2}{\varepsilon\sqrt{n}}\sum_{k=1}^{n}\Bigg\{\bigg(\frac{\partial^2}{\partial\alpha_i\partial\beta_j}b_{t_{k-1},H_n}(\theta)\bigg)^{\top}\Xi_{k-1}^{-1}(\beta)+\bigg(\frac{\partial}{\partial\alpha_i}b_{t_{k-1},H_n}(\theta)\bigg)^{\top}\bigg(\frac{\partial}{\partial\beta_j}\Xi_{k-1}^{-1}(\beta)\bigg)\Bigg\}\notag\\
&\quad \cdot\Big\{P_k(\theta_0)+\frac{1}{n}\Big(b_{t_{k-1},H_n}(\theta_0)-b_{t_{k-1},H_n}(\theta)\Big)\Big\}\notag\\
&+\frac{2}{\varepsilon n\sqrt{n}}\sum_{k=1}^{n}\bigg(\frac{\partial}{\partial\alpha_i}b_{t_{k-1},H_n}(\theta)\bigg)^{\top}\Xi_{k-1}^{-1}(\beta)\bigg(\frac{\partial}{\partial\beta_j}b_{t_{k-1},H_n}(\theta)\bigg),\label{alphabeta}
\end{align}
\begin{align}
&\frac{1}{n}\frac{\partial^2}{\partial\beta_i\partial\beta_j}U_{n,\varepsilon}(\theta)\notag\\
= &\frac{1}{n}\sum_{k=1}^{n}\frac{\partial^2}{\partial\beta_i\partial\beta_j}\log\det\Xi_{k-1}(\beta) -\frac{2\varepsilon^{-1}}{n}\sum_{k=1}^{n}\Bigg\{\bigg(\frac{\partial^2}{\partial\beta_i\partial\beta_j}b_{t_{k-1},H_n}(\theta)\bigg)^{\top}\Xi_{k-1}^{-1}(\beta)\notag\\
&\quad +\bigg(\frac{\partial}{\partial\beta_i}b_{t_{k-1},H_n}(\theta)\bigg)^{\top}\bigg(\frac{\partial}{\partial\beta_j}\Xi_{k-1}^{-1}(\beta)\bigg) +\bigg(\frac{\partial}{\partial\beta_j}b_{t_{k-1},H_n}(\theta)\bigg)^{\top}\bigg(\frac{\partial}{\partial\beta_i}\Xi_{k-1}^{-1}(\beta)\bigg)\Bigg\}\notag\\
&\quad \cdot\Big\{P_k(\theta_0)+\frac{1}{n}\Big(b_{t_{k-1},H_n}(\theta_0)-b_{t_{k-1},H_n}(\theta)\Big)\Big\}\notag\\
+&\frac{2\varepsilon^{-1}}{n^2}\sum_{k=1}^{n}\bigg(\frac{\partial}{\partial\beta_i}b_{t_{k-1},H_n}(\theta)\bigg)^{\top}\Xi_{k-1}^{-1}(\beta)\bigg(\frac{\partial}{\partial\beta_j}b_{t_{k-1},H_n}(\theta)\bigg)\notag\\
+&\varepsilon^{-2}\sum_{k=1}^{n}\big(P_k(\theta)\big)^{\top}\bigg(\frac{\partial^2}{\partial\beta_i\partial\beta_j}\Xi_{k-1}^{-1}(\beta)\bigg)P_k(\theta).\label{betabeta}
\end{align}
It follows from Lemma \ref{convergence depend on n} and Lemma \ref{convergence of diffusion}(ii), as $\varepsilon\rightarrow 0$ and $n\rightarrow \infty,$ that
\begin{flalign*}
&\varepsilon^2\frac{\partial^2}{\partial\alpha_i\partial\alpha_j}U_{n,\varepsilon}(\theta)&\\
\xrightarrow{P}& 2\int_0^1\bigg(\frac{\partial}{\partial\alpha_i}b\left(X_{s}^0,H(X_{s-\cdot}^0),\theta\right)\bigg)^{\top}\left[\sigma\sigma^{\top}\right]^{-1}\left(X_s^0,H(X_{s-\cdot}^0),\beta\right)\bigg(\frac{\partial}{\partial\alpha_j}b\left(X_{s}^0,H(X_{s-\cdot}^0),\theta\right)\bigg)\,\df s&\\
&-2\int_0^1\bigg(\frac{\partial^2}{\partial\alpha_i\partial\alpha_j}b(X_{s}^0,H(X_{s-\cdot}^0),\theta)\bigg)^{\top}\left[\sigma\sigma^{\top}\right]^{-1}\left(X_s^0,H(X_{s-\cdot}^0),\beta\right)B\left(X_s^0, \theta_0, \theta\right)\,\df s,&
\end{flalign*}
\begin{flalign*}
&\frac{\varepsilon}{\sqrt{n}}\frac{\partial^2}{\partial\alpha_i\partial\beta_j}U_{n,\varepsilon}(\theta)\xrightarrow{P}0,&
\end{flalign*}
\begin{flalign*}
&\frac{1}{n}\frac{\partial^2}{\partial\beta_i\partial\beta_j}U_{n,\varepsilon}(\theta)&\\
\xrightarrow{P}&\int_0^1\frac{\partial^2}{\partial\beta_i\partial\beta_j}\log\det\left[\sigma\sigma^{\top}\right]\left(X_s^0,H(X_{s-\cdot}^0),\beta\right)\,\df s&\\
&+\int_0^1\mathrm{tr}\Bigg[\left[\sigma\sigma^{\top}\right]\left(X_s^0,H(X_{s-\cdot}^0),\beta\right)\bigg(\frac{\partial^2}{\partial\beta_i\partial\beta_j}\left[\sigma\sigma^{\top}\right]^{-1}\left(X_s^0,H(X_{s-\cdot}^0),\beta\right)\bigg)\Bigg]\,\df s,&
\end{flalign*}
uniformly in $\theta\in\Theta$. About \eqref{check2 of a.n.}, the limit of \eqref{alphaalpha}, \eqref{alphabeta} and \eqref{betabeta} are continuous with respect to $\theta$, which completes the proof. Finally, we prove \eqref{check3 of a.n.}. We set
\begin{align*}
-\varepsilon\frac{\partial}{\partial \alpha_{i}}U_{n,\varepsilon}(\theta)\bigg|_{\theta=\theta_{0}}&= \sum_{k=1}^{n}2\varepsilon^{-1}\left(\frac{\partial}{\partial \alpha_{i}}b_{t_{k-1},H_n}\bigl(\theta_0\bigr)\right)^{\top}\Xi_{k-1}^{-1}(\beta_0)P_{k}(\theta_0)\\
&=\colon \sum_{k=1}^{n}\xi_{k}^{i}(\theta_{0}),
\end{align*}
\begin{align*}
-\frac{1}{\sqrt{n}}\frac{\partial}{\partial \beta_{i}}U_{n,\varepsilon}(\theta)\bigg|_{\theta=\theta_{0}}&= -\sum_{k=1}^{n}\frac{1}{\sqrt{n}}\frac{\partial}{\partial \beta_{i}}\log\det\Xi_{k-1}(\beta_0)\\
&\quad-\sum_{k=1}^{n}\sqrt{n}\varepsilon^{-2}\Bigl(P_{k}(\theta_{0})\Bigr)^{\top}\left(\frac{\partial}{\partial \beta_{i}}\Xi_{k-1}^{-1}(\beta_0)\right)P_{k}(\theta_{0})\\
&\quad +2\sum_{k=1}^{n}\frac{\varepsilon^{-2}}{\sqrt{n}}\left(\frac{\partial}{\partial \beta_{i}}b_{t_{k-1},H_n}\bigl(\theta_{0}\bigr)\right)^{\top}\Xi_{k-1}^{-1}(\beta_0)P_{k}(\theta_{0})\\
&=\colon \sum_{k=1}^{n}\eta_{k}^{i}(\theta_{0}).
\end{align*}
In view of Theorem 3.2 and 3.4 in Hall and Heyde \cite{hall}, it is sufficient to show that as $\varepsilon\to0$ and $n\to\infty$,
\begin{align}
&\sum_{k=1}^{n}E\left[\xi_{k}^{i}(\theta_{0})\Big|\mathcal{F}_{t_{k-1}}\right]\xrightarrow{P}0,\label{hall1}\\
&\sum_{k=1}^{n}E\left[\eta_{k}^{j}(\theta_{0})\Big|\mathcal{F}_{t_{k-1}}\right]\xrightarrow{P}0,\\
&\sum_{k=1}^{n}E\left[\xi_{k}^{i_{1}}(\theta_{0})\xi_{k}^{i_{2}}(\theta_{0})\Big|\mathcal{F}_{t_{k-1}}\right]\xrightarrow{P}4I_{b}^{i_{1}i_{2}}(\theta_{0}),\\
&\sum_{k=1}^{n}E\left[\eta_{k}^{j_{1}}(\theta_{0})\eta_{k}^{j_{2}}(\theta_{0})\Big|\mathcal{F}_{t_{k-1}}\right]\xrightarrow{P}4I_{\sigma}^{j_{1}j_{2}}(\theta_{0}),\\
&\sum_{k=1}^{n}E\left[\xi_{k}^{i}(\theta_{0})\eta_{k}^{j}(\theta_{0})\Big|\mathcal{F}_{t_{k-1}}\right]\xrightarrow{P}0,\\
&\sum_{k=1}^{n}E\left[(\xi_{k}^{i}(\theta_{0}))^{4}\Big|\mathcal{F}_{t_{k-1}}\right]\xrightarrow{P}0,\label{hall2}\\
&\sum_{k=1}^{n}E\left[(\eta_{k}^{j}(\theta_{0}))^{4}\Big|\mathcal{F}_{t_{k-1}}\right]\xrightarrow{P}0.\label{hall3}
\end{align}
From Lemma \ref{conditional expectation}, we obtain \eqref{hall1}-\eqref{hall2}. To prove \eqref{hall3}, we have several estimates as follows:
\begin{align}
&\big(\eta_{k}^{j}(\theta_{0})\big)^4\notag\\
\leq& 3^3\Bigg\{\frac{1}{n^2}\bigg(\frac{\partial}{\partial \beta_{j}}\log\det\Xi_{k-1}(\beta_0)\bigg)^4+n^2\varepsilon^{-8}(2d)^3\sum_{l_1,l_2}^{d}\Bigg[\bigg(\frac{\partial}{\partial\beta_j}\Xi_{k-1}^{-1}(\beta_0)\bigg)^{l_1l_2}\Bigg]^4\big(P_k^{l_1}P_k^{l_2}(\theta_0)\big)^4\notag\\
& +16n^{-2}\varepsilon^{-8}d^3\sum_{l_1}^{d}\Bigg[\bigg\{\biggl(\frac{\partial}{\partial \beta_{j}}b_{t_{k-1}, H_n}\bigl(\theta_{0}\bigr)\biggr)^{\top}\Xi_{k-1}^{-1}(\beta_0)\bigg\}^{l_1}\Bigg]^4\big(P_{k}^{l_1}(\theta_{0})\big)^4\Bigg\},\label{eta1}
\end{align}
\begin{align}
E\Big[\big(P_k^{l_1}P_k^{l_2}(\theta_0)\big)^4\Big|\mathcal{F}_{t_{k-1}}\Big]&\leq 3^3\bigg\{E\Big[\Big(\big(\Delta_k X\big)^{l_1}\big(\Delta_k X\big)^{l_2}\Big)^4\Big|\mathcal{F}_{t_{k-1}}\Big]\notag\\
&\quad +\frac{1}{n^4}\Big(b_{t_{k-1}, H_n}^{l_1}\big(\theta_{0}\big)\Big)^4E\Big[\Big(\big(\Delta_k X\big)^{l_2}\Big)^4\Big|\mathcal{F}_{t_{k-1}}\Big]\notag\\
&\quad +\frac{1}{n^4}\Big(b_{t_{k-1}, H_n}^{l_2}\big(\theta_{0}\big)\Big)^4E\Big[\big(P_k^{l_1}(\theta_0)\big)^4\Big|\mathcal{F}_{t_{k-1}}\Big]\bigg\}.\label{eta2}
\end{align}
In the same way as the proof of Lemma \ref{conditional expectation}, we have
\begin{align}
&E\bigg[\prod_{i=1}^{8}\Big(\Delta_k X^{\varepsilon}\Big)^{l_i}\bigg|\mathcal{F}_{t_{k-1}}\bigg]\notag\\
= &R_{k-1}\bigg(\frac{1}{n^8}\bigg)+R_{k-1}\bigg(\frac{\varepsilon^2}{n^7}\bigg)+R_{k-1}\bigg(\frac{\varepsilon^4}{n^6}\bigg)+R_{k-1}\bigg(\frac{\varepsilon^6}{n^5}\bigg)+R_{k-1}\bigg(\frac{\varepsilon^8}{n^4}\bigg)\notag\\
& +\int_{t_{k-1}}^{t_k}\{\Phi_8^\varepsilon(s_1)+\varepsilon^2\Phi_7^\varepsilon(s_1)\}\,\df s_1\notag\\
& +\sum_{i=1}^{6}\sum_{j=0}^{6-i}R_{k-1}(\varepsilon^{2j})\int_{t_{k-1}}^{t_k}\int_{t_{k-1}}^{s_1}\cdots\int_{t_{k-1}}^{s_i}\{\Phi_{8-i-j}^\varepsilon(s_{i+1})+\varepsilon^2\Phi_{7-i-j}^\varepsilon(s_{i+1})\}\,\df s_{i+1}\cdots\,\df s_1\notag\\
& +R_{k-1}(1)\int_{t_{k-1}}^{t_k}\int_{t_{k-1}}^{s_1}\cdots\int_{t_{k-1}}^{s_7}\Phi_1^\varepsilon(s_8)\,\df s_8\cdots\,\df s_1\notag\\
& +R_{k-1}\bigg(\frac{1}{n^9}\bigg)+R_{k-1}\bigg(\frac{\varepsilon}{n^8\sqrt{n}}\bigg)+R_{k-1}\bigg(\frac{\varepsilon^2}{n^8}\bigg)+R_{k-1}\bigg(\frac{\varepsilon^3}{n^7\sqrt{n}}\bigg)+R_{k-1}\bigg(\frac{\varepsilon^4}{n^7}\bigg)\notag\\
& +R_{k-1}\bigg(\frac{\varepsilon^5}{n^6\sqrt{n}}\bigg)+R_{k-1}\bigg(\frac{\varepsilon^6}{n^6}\bigg)+R_{k-1}\bigg(\frac{\varepsilon^7}{n^5\sqrt{n}}\bigg)+R_{k-1}\bigg(\frac{\varepsilon^8}{n^5}\bigg)+R_{k-1}\bigg(\frac{\varepsilon^9}{n^4\sqrt{n}}\bigg).\label{eta3}
\end{align}
It follows from Lemma \ref{conditional expectation}, \eqref{eta1}-\eqref{eta3} and the H\"{o}lder's inequality that
\begin{align*}
\sum_{k=1}^{n}E\left[(\eta_{k}^{j}(\theta_{0}))^{4}\Big|\mathcal{F}_{t_{k-1}}\right]\xrightarrow{P} 0,
\end{align*}
as $\varepsilon\to0$ and $n\to\infty$. We obtain the conclusion. 
\end{proof}

\begin{flushleft}
{\bf Compliance with Ethical Standards.} 
\end{flushleft}
\begin{itemize}
\item Funding: The study was funded by JSPS  KAKENHI Grant Number 21K03358. 

\item Conflict of Interest: Author Hiroki Nemoto declares that he has no conflict of interest.
Author Yasutaka Shimizu has received research grants from JSPS. 

\item Ethical approval: This article does not contain any studies with human participants performed by any of the authors.
\end{itemize}

\end{document}